\documentclass[11pt]{amsart}
\usepackage{amssymb,amscd,  latexsym, graphicx, mathrsfs, enumerate}
\usepackage{xcolor}

\usepackage{caption}

\newcommand{\ord}{{\rm ord}}
\newcommand{\nc}{\newcommand}
\numberwithin{equation}{section}
\newtheorem{theorem}{Theorem}[section]
\newtheorem{prop}[theorem]{Proposition}
\newtheorem{importnota}[theorem]{Important Notation}
\newtheorem{prblm}[theorem]{Problem}
\newtheorem{notation}[theorem]{Notation}
\newtheorem{caution}[theorem]{Caution}
\newtheorem{remark}[theorem]{Remark}
\newtheorem{lemma}[theorem]{Lemma}
\newtheorem{construction}[theorem]{Construction}
\newtheorem{corollary}[theorem]{Corollary}
\newtheorem{example}[theorem]{Example}
\newtheorem{conclusion}[theorem]{Conclusion}
\newtheorem{triviality}[theorem]{Triviality}
\newtheorem{proto}[theorem]{Prototype Quasifibration}
\newtheorem{cauex}[theorem]{Cautionary Example}
\newtheorem{propositiondef}[theorem]{Proposition-Definition}
\newtheorem{subth}{Nuisance}[theorem]
\newtheorem{ssubth}{ }[subth]
\newtheorem{conjecture}[theorem]{Conjecture}
\newtheorem{sidest}[theorem]{Side Story}
\newtheorem{miniexample}[theorem]{Example}
\theoremstyle{definition}
\newtheorem{defin}[theorem]{Definition}

\nc\tri[1]{\begin{triviality}}
\nc\side[1]{\begin{sidest}}
\nc\conj[1]{\begin{conjecture}}
\nc\prodef[1]{\begin{propositiondef}}
\nc\prt[1]{\begin{proto}}
\nc\lem[1]{\begin{lemma}}
\nc\sblm[1]{\begin{sublemma}}
\nc\pro[1]{\begin{prop}}
\nc\thm[1]{\begin{theorem}}
\nc\cor[1]{\begin{corollary}}
\nc\dfn[1]{\begin{defin}}
\nc\sthm[1]{\begin{subth}}
\nc\exm[1]{\begin{example}}
\nc\miniexm[1]{\begin{miniexample}}
\nc\plm[1]{\begin{prblm}}
\nc\rmk[1]{\begin{remark}}
\nc\subrmk[1]{\begin{subremark}}
\nc\ntn[1]{\begin{notation}}
\nc\cau[1]{\begin{caution}}
\nc\imn[1]{\begin{importnota}}
\nc\cax[1]{\begin{cauex}}
\nc\con[1]{\begin{construction}}
\nc\ssthm[1]{\begin{ssubth}}
\nc\cnc[1]{\begin{conclusion}}
\nc\elem{\end{lemma}}
\nc\esblm{\end{sublemma}}
\nc\eside{\end{sidest}}
\nc\econj{\end{conjecture}}
\nc\eprodef{\end{propositiondef}}
\nc\eprt{\end{proto}}
\nc\ethm{\end{theorem}}
\nc\ecor{\end{corollary}}
\nc\edfn{\end{defin}}
\nc\esthm{\end{subth}}
\nc\epro{\end{prop}}
\nc\etri{\end{triviality}}
\nc\eexm{\end{example}}
\nc\eminiexm{\end{miniexample}}
\nc\ermk{\end{remark}}
\nc\subermk{\end{subremark}}
\nc\eplm{\end{prblm}}
\nc\ecau{\end{caution}}
\nc\ecax{\end{cauex}}
\nc\eimn{\end{importnota}}
\nc\entn{\end{notation}}
\nc\econ{\end{construction}}
\nc\ecnc{\end{conclusion}}
\nc\essthm{\end{ssubth}}

\newcommand{\C}{\mathbb{C}}
\newcommand{\R}{\mathbb{R}}
\newcommand{\Q}{\mathbb{Q}}

\newcommand{\Z}{\mathbb{Z}}

\newcommand{\A}{\mathbb{A}}

\newcommand{\diag}{{\rm diag}}

\renewcommand{\o}{\mathfrak{o}}
\newcommand{\ds}{\displaystyle}

\newcommand{\lra}{\longrightarrow}
\newcommand{\bs}{\backslash}

\newcommand{\sgn}{{\rm sgn}}

\renewcommand{\Bbb}{\mathbb}
\def\GL{{\mathop{\mathrm{GL}}}}

\def\SL{{\mathop{\mathrm{SL}}}}

\def\Sp{{\mathop{\mathrm{Sp}}}}

\def\Sym{{\mathop{\mathrm{Sym}}}}

\begin{document}

\thispagestyle{empty}

\title[Restriction of modular forms on $E_{7,3}$ to $Sp_6$]{Restriction of modular forms on $E_{7,3}$ to $Sp_6$}
\author{Henry H. Kim and Takuya Yamauchi}
\date{\today}
\thanks{The first author is partially supported by NSERC grant \#482564. 
}
\address{Henry H. Kim \\
Department of mathematics \\
 University of Toronto \\
Toronto, Ontario M5S 2E4, CANADA \\
and Korea Institute for Advanced Study, Seoul, KOREA}
\email{henrykim@math.toronto.edu}

\address{Takuya Yamauchi \\
Mathematical Inst. Tohoku Univ.\\
 6-3,Aoba, Aramaki, Aoba-Ku, Sendai 980-8578, JAPAN}
\email{takuya.yamauchi.c3@tohoku.ac.jp}


\maketitle

\begin{abstract} 
In this paper, we study the restriction of modular forms such as Ikeda type lifts and the  
Eisenstein series on the exceptional group of type $E_{7,3}$ to the symplectic group $Sp_6$ (rank 3). 
As an application, we explicitly write down the restriction when modular forms 
have small weight. The restriction may contain Miyawaki lifts of type I,II (CAP forms) and 
genuine forms whose description is compatible with Arthur's classification.
\end{abstract}

\tableofcontents

\section{Introduction}
Let $\A$ be the ring of adeles of $\Q$. 
Let $G$ be a reductive group over $\Q$ and $H$ a subgroup of $G$ which is 
also defined over $\Q$. The restriction of an 
automorphic form on $G(\A)$ to $H(\A)$ has been an interesting object as in 
the Gan-Gross-Prasad conjecture.  The block restriction or diagonal restriction of 
Siegel modular forms has been a useful way to study the graded ring of Siegel modular 
forms or to construct Siegel modular forms from known forms (cf. \cite{CvdG}). 
The first important question is whether or not the restriction is non-vanishing. Then, if it is non-vanishing, we would be concerned with describing it in terms of automorphic forms on 
$H(\A)$. 

In this paper, we study the restriction of 
holomorphic modular forms on the exceptional group of type $E_{7,3}$ to 
the symplectic group $Sp_6$ of rank 3. 
It is known that $G^{\rm c}_2\times Sp_6\subset E_{7,3}$ is an exceptional dual pair 
where $G^{\rm c}_2$ is the anisotropic form of $G_2$ (cf. \cite{GS}). 
Thus, for any modular form $F$ on $E_{7,3}$ of level one, $F|_{Sp_6}$ can be regarded 
as the integral of the product of the constant function $1_{G^{\rm c}_2}$ on $G^{\rm c}_2$ and 
$F|_{G^c_2\times Sp_6}$ against 
$G^{\rm c}_2(\A)$. If $F$ is an Ikeda type lift constructed in \cite{KY}, 
any non-archimedean component of the cuspidal representation $\Pi_F$ attached to $F$ is 
a degenerate principal series for the Siegel parabolic subgroup $P$ inside $E_{7,3}$. 
Thus, the restriction $F|_{Sp_6}$ is similar to the setting of the exceptional theta 
correspondence (cf. \cite{MS, GS}) or Miyawaki-type construction \cite{Ik,KY1,Atobe}. 
However, the double coset space $P\bs G/(G^c_2\times Sp_6)$ may not be finite by \cite{D} 
and thus, we may not expect a usual doubling method to work in our setting. Further, for each prime 
$p$, $\Pi_{F,p}$ is far from the minimal representation, and hence we cannot use the techniques in \cite{MS, GS}. 
Instead, in this paper, we study the restriction by partially computing the Fourier 
coefficients $A_F(T)$ of $F$ when $F$ is an Ikeda type lift or an Eisenstein series 
for the indices $T$ with a particular shape. Reflecting the facts that $\Pi_{F,p}$ is not the minimal representation, and the double coset $P\bs G/(G^c_2\times Sp_6)$ has positive dimension, the restriction of each Ikeda type lift may contain various components including genuine forms. 
In particular, the restriction does not preserve Hecke eigenforms. This is indeed observed in Section 6. 

We organize our paper as follows. In Section 2, we reformulate the 
restriction problem in terms of the classical language, namely, we recall 
the definition of the modular forms on the exceptional domain and its 
restriction to the Siegel upper half-plane of degree 3. 
In Section 3, we compute the number of imaginary octonions which 
satisfy various properties. It is of independent interest. 
Several of these results will be used in Section 6. 
In Section 4, we recall an explicit form of the Fourier 
coefficients of Eisenstein series and Ikeda type lift on $\frak T$. 
In Section 5, we review Siegel modular forms of degree 3 and level one from Yuen-Poor-Shurman \cite{YPS, KPSY}, and 
in Section 6, we compute the restriction of the Ikeda type lift of weight 20 on $E_{7,3}$ to $Sp_6$.
We also compute the restriction of the Eisenstein series of low weights on $E_{7,3}$ to $Sp_6$.
In the Appendix A, we compute all 
$T\in \frak J_{>0}$ such that the restriction to half-integral matrices is $\diag(2,2,2)$. In the Appendix B, we show, by using Arthur's classification, that all 
Hecke eigen holomorphic Siegel modular forms of degree 3 with level 1 and the scalar weight $k\ge 4$ are either of Miyawaki lift of type I, II, or genuine forms. In particular, no endoscopic forms show up. 
In the course of the proof, we also describe cuspidal representations attached to 
Miyawaki lifts in terms of CAP representations. 

\textbf{Acknowledgments.} We would like to thank Cris Poor for helpful discussions and his encouragement. We thank the referees for many useful comments and corrections.

\section{Review of the exceptional domain}\label{excep}
We will freely use the notations from \cite[Section 2]{KY} (see also \cite{Ba, Kim}). 
Let $\frak C_{\Bbb Q}=\bigoplus_{i=0}^7 \Q e_i$ and $\mathfrak o\subset \frak C_{\Bbb Q}$ be the Cayley numbers (octonions) and integral Cayley numbers (integral octonions), resp. 
For any $\Q$-algebra $R$, an element $x\in \frak C_R:=\frak C_{\Bbb Q}\otimes_\Q R$ is said to be imaginary or imaginary octonion if the coefficient of $e_0$ in $x$ is zero.  
For each $x=\ds\sum_{i=0}^7 x_i e_i\in \frak C_R$, define the norm and the trace of $x$ by 
$$N(x):=\sum_{i=0}^7x^2_i,\ {\rm Tr}(x)=2x_0$$
respectively.
Let $\frak J_\Bbb Q$ be the exceptional Jordan algebra consisting of matrices 
\[X=(x_{ij})_{1 \le i,j \le 3}=
\begin{pmatrix}
      a  & x       & y \\
\bar x & b        & z \\ 
\bar y & \bar z & c 
\end{pmatrix}\]
with $a,b,c \in \Bbb Q$ and $x,y,z \in \frak C_{\Bbb Q}$. 
We define the Jordan algebra product $X\times X$ by 
\begin{equation}\label{JAP}
X\times X=
\begin{pmatrix}
      bc-N(z)  & y\bar{z}-cx       & xz-by \\
z\bar{y}-c\bar{x} & ac-N(y)        & \bar{x}{y}-az \\ 
\bar{z}\bar{x}-b\bar{y} & \bar{y}x-a\bar{z} & ab-N(x) 
\end{pmatrix}. 
\end{equation}
We define the determinant $\det X$ and the trace $\mathrm{Tr}(X)$ by
\[\det X=abc-aN(z)-bN(y)-cN(x)+\mathrm{Tr}((xz)\bar y), \quad  \mathrm{Tr}(X)=a+b+c.\]
We define a lattice $\frak J(\Bbb Z)$ of $\frak J_{\Bbb Q}$  by 
\[\frak J(\Bbb Z)=\{ X =(x_{ij}) \in \frak J_{\Bbb Q} \ | \ x_{ii} \in \Bbb Z \text{ and } x_{ij} \in \frak o \text{ for } i\not=j\}. \]
For a commutative algebra $R$, we put $\frak J(R)=\frak J(\Bbb Z) \otimes_{\Bbb Z} R$. 
Recall
$$\frak J(R)^{\rm ns}=\{ X \in \frak J(R) \ | \ \det (X) \not=0\},\quad R^+_3(R)=\{ X^2 \ | \ X \in \frak J(R)^{\rm ns}\}.
$$
We denote by $\overline{R^+_3(\Bbb R)}$ the closure of $R^+_3(\Bbb R)$ in $\frak J(\Bbb R) \simeq \Bbb R^{27}$. For a subring $A$ of $\Bbb R$, set
\[\frak J(A)_{>0}=\frak J(A) \cap R^+_3(\Bbb R) \text{ and } \frak J(A)_{\ge 0}=\frak J(A) \cap \overline{R^+_3(\Bbb R)}.\]
For a given $T=\begin{pmatrix}
      a  & x       & y \\
\bar x & b        & z \\ 
\bar y & \bar z & c 
\end{pmatrix}\in \frak J(A)$, $T\in \frak J(A)_{\ge 0 }$ if and only if 
$$a\ge 0,\ b\ge 0,\ c\ge 0, \ \det(T)\ge 0,\ ab-N(x)\ge 0,\ ac-N(y)\ge 0,\ bc-N(z)\ge 0.$$
Similarly, $T\in \frak J(A)_{>0 }$ if and only if 
$$a>0,\ b>0,\ c>0,\ \det(T)>0,\ ab-N(x)>0,\ ac-N(y)>0,\ bc-N(z)>0.$$

Then the exceptional domain is given by
$$\frak T:=\{Z=X+Y\sqrt{-1}\in \frak J_\Bbb C\ |\ X,Y\in \frak J_\Bbb R,\ Y\in R^+_3(\Bbb R)\}
$$
which is a complex analytic subspace of $\Bbb C^{27}$ and 
it is the Hermitian symmetric space for the exceptional group of type 
$E_{7,3}$. 

Let $\Bbb H_3$ be the Siegel upper half space of degree 3:
$$\Bbb H_3=\{ Z_1=X_1+Y_1\sqrt{-1}\in M_3(\Bbb C) : \text{$X_1, Y_1$ real,\, $Z_1={}^t Z_1$, \, $Y>0$}\},
$$
which is the Hermitian symmetric space for the symplectic group $Sp_6$ (of rank 3). 
It is known that $Sp_6$ can be regarded as a subgroup of $E_{7,3}$ and it yields a natural embedding from the Siegel upper half space into the exceptional domain: $\Bbb H_3\hookrightarrow \frak T,\ Z_1\mapsto Z_1$ so that 
the actions of $\Sp_6(\R)$ and $E_{7,3}(\R)$ on $\Bbb H_3$ and $\frak T$, respectively, are compatible with the embedding $Sp_6\subset E_{7,3}$. 

Each $Z\in\frak T$ can be written uniquely as $Z=Z_1+Z_2$, where $Z_1=
\begin{pmatrix}
\tau_1&z_1&z_2
\\ z_1&\tau_2&z_3
\\ z_2&z_3&\tau_3
\end{pmatrix}\in \Bbb H_3$, and
$Z_2=
\begin{pmatrix}
0&w_1&w_2\\ 
-w_1&0&w_3\\ 
-w_2&-w_3&0
\end{pmatrix}\in \frak J_\C$, where $w_1,w_2,w_3$ are imaginary octonions so that $\overline{w_i}=-w_i$. We call $Z_2$ the imaginary octonion part of $Z$. Similarly we can write 
$T\in\frak J(\Bbb Z)$ as $T=T_1+T_2$, where $T_1\in {\rm Sym}^3(\Bbb Z)_{\ge 0}$ is a semi-definite half-integral symmetric matrix and $T_2$ is the imaginary octonion part of $T_2$.

A modular form $F$ on $\frak T$ of weight $k$ is 
a holomorphic function on $\frak T$ which satisfies
$$F(\gamma Z)=F(Z) j(\gamma,Z)^k,\quad Z\in \frak T,\, \gamma\in\Gamma,\quad \gamma\in \Gamma=E_{7,3}(\Z),
$$
where $j(\gamma,Z)$ is the canonical factor of automorphy as in \cite[p. 234]{KY}. Such $F$ has a Fourier expansion of the form
\begin{equation*}
F(Z)=\sum_{T\in \frak J(\Z)_{\geq 0}} a(T) e^{2\pi i (T,Z)},
\end{equation*}
where $(X,Y)={\rm Tr}(X\circ Y)$, and $X\circ Y=\frac 12(X\cdot Y+Y\cdot X)$, and $X\cdot Y$ is the matrix multiplication.

Let $\bold M'(\Bbb Z)$ be the subgroup of $\Gamma$ as in \cite[p. 227]{KY}. 
It plays a role of $\begin{pmatrix} U&0\\0&{}^t U^{-1}\end{pmatrix}\in \Sp_{6}(\Bbb Z)$ with $U\in \GL_3(\Bbb Z)$. We say that $T,T'\in \frak J(\Bbb Z)_{\ge 0}$ are equivalent over $\Bbb Z$ if $T'=m(T)$ for some $m\in \bold M'(\Bbb Z)$.
If $F$ is a modular form with respect to $\Gamma$, since $F(m Z)=F(Z)$, $a(T')=a(T)$ if $T,T'$ are equivalent over $\Bbb Z$.

We say $T,T'\in \frak J(\Bbb Z)_{\ge 0}$ are equivalent over $\Bbb Z_p$ if $T'=m(T)$ for some $m\in \bold M'(\Bbb Z_p)$. By \cite[p. 521]{Ba}, any $T\in \frak J(\Bbb Z)_{\ge 0}$ is equivalent to a diagonal matrix over $\Bbb Z_p$.

Similarly, a modular form $F$ on $\Bbb H_3$ of weight $k$ is 
a holomorphic function on $\Bbb H_3$ which satisfies
$$F(\gamma Z)=\det(CZ+D)^k F(Z),\quad \gamma=\begin{pmatrix} A&B\\C&D\end{pmatrix}\in\Gamma_3=\Sp_6(\Z),
$$
Such $F$ has a Fourier expansion of the form
\begin{equation*}
F(Z)=\sum_{T\in \Sym^3(\Bbb Z)_{\geq 0}} a(T) e^{2\pi i{\rm Tr}(TZ)},
\end{equation*}
where $\Sym^3(\Bbb Z)_{\geq 0}$ is the set of half-integral $3\times 3$ symmetric matrices over $\Bbb Z$.

\begin{lemma} For a holomorphic modular form $F$ of weight $k$ on $\frak T$, we denote by $F|_{Sp_6}$, its restriction to $\Bbb H_3$. Then
 $F|_{Sp_6}$ is a holomorphic 
Siegel modular form on $\Bbb H_3$ of weight $k$, and it is of level 1 if $F$ is. 
Here the level 1 means that the forms have the levels with respect to 
$\Sp_6(\Z)$ and $E_{7,3}(\Z)$, respectively. 

If $F$ is a cusp form 
of level one, then 
$F|_{Sp_6}$ is a cusp form of level one on $\Bbb H_3$.
\end{lemma}
\begin{proof} Since $Z_1\to -Z_1^{-1}$ and $Z_1\to Z_1+S_1$, $S_1\in \Sym^3(\Bbb Z)\geq 0$, generate $\Gamma_3$, it is enough to check $F|_{Sp_6}(-Z^{-1}_1)=\det(-Z_1)^k F|_{Sp_6}(Z_1)$ and $F|_{Sp_6}(Z_1+S_1)=F|_{Sp_6}(Z_1)$. 
Likewise, $\Gamma$ is generated by $\bold N(\Bbb Z)$ and $\iota$ \cite[Lemma 5.1]{KY}. By \cite[p. 527]{Ba}, $\iota Z=-Z^{-1}$ and $j(\iota,Z)=\det(-Z)$, and since $F(\iota Z)=\det(-Z)^k F(Z)$, we have $F(-Z_1^{-1})=\det(-Z_1)^k F(Z_1)$. 

Let $\Phi_{\Sp_6}$ be the Siegel phi-operator for $Sp_6$ (cf. \cite[Definition 4, 
p.192]{vdGeer}) and $\Phi$ the operator defined in \cite[p. 234]{KY}.  
Then, for any $Z^{(2\times 2)}_1\in \mathbb{H}_2$, 
$$\Phi_{\Sp_6}(F|_{Sp_6})(Z^{(2\times 2)}_1)=
\lim_{\tau\in \mathbb{H}_1\atop \tau\to \sqrt{-1}\infty}F(
\begin{pmatrix}
Z^{(2\times 2)}_1&0 \\
0 & \tau
\end{pmatrix})
=\Phi(F)(Z^{(2\times 2)}_1)=0
$$
since $\Phi(F)=0$ by cuspidality of $F$. 
Here, $Z^{(2\times 2)}_1$ is naturally regarded as an element of the 
exceptional domain $\frak T_2$ for $Spin(2,10)$ (see \cite[Section 2]{KY}). 
\end{proof}

 \section{Number of imaginary octonions of a given norm}
In this section, we introduce some important facts on integral octonions and imaginary octonions. 
However, they are easily deduced from the definition and therefore, we omit the details. 
 
Let $\mathfrak o$ be the set of integral octonions given by the following basis:
$$\alpha_0=e_0,\quad \alpha_1=e_1,\quad \alpha_2=e_2,\quad \alpha_3=-e_4,\quad \alpha_4=\frac 12(e_1+e_2+e_3+e_4),
$$
$$\alpha_5=\frac 12(-e_0-e_1-e_4+e_5),\quad \alpha_6=\frac 12(-e_0+e_1-e_2+e_6),\quad \alpha_7=\frac 12(-e_0+e_2+e_4+e_7).
$$

\begin{lemma}\label{oct} Integral octonions of norm 1 are exactly roots of the exceptional $E_8$ root system. Hence integral octonions coincide with elements of the root lattice of $E_8$.
\end{lemma}

\begin{remark} Note that the lattices $\frak o$ and $\frac 1{\sqrt{2}}E_8$ are isometric, and so the integral octonions of norm one correspond to the elements in the root system of norm two. Let $Q_8$ be the root lattice of $E_8$. Then the norms on $\frak o$ and on $Q_8$ are both denoted by $N$.
\end{remark}

Consider
$$\theta_{Q_8}(z)=\sum_{\alpha\in Q_8} q^{\frac 12 N(\alpha)}=1+\sum_{n=1}^\infty 
a_n(Q_8) q^n,
$$
where $q=e^{2\pi i z}$ with $z\in \mathbb{H}_1:=\{z\in \C\ |\ {\rm Im}(z)>0\}$. Given a positive integer $n$, let 
$$N_{{\rm oc}}(n)=\#\{ \text{$w$: integral octonion such that $N(w)=n$}\}.$$
Then $a_n(Q_8)=N_{{\rm oc}}(n)$. Since $\theta_{Q_8}$ is a modular form of weight 4 and dim$M_4(\SL_2(\Bbb Z))=1$, $\theta_{Q_8}$ is the Eisenstein series of weight 4 
with respect to $\SL_2(\Z)$, i.e.,
$$\theta_{Q_8}(z)=1+240\sum_{n=1}^\infty \sigma_3(n)q^n,\ \sigma_3(n)=\ds\sum_{d|n,\ d\ge 1}d^3. 
$$
Hence we have 
\begin{corollary} $N_{{\rm oc}}(n)=240\sigma_3(n)$.
\end{corollary}

Let $\frak o'$ be the set of all imaginary integral octonions.

\begin{lemma}\cite[p.241]{B} Integral imaginary octonions of norm 1 are exactly roots of the exceptional $E_7$ root system. Hence imaginary integral octonions coincide with elements of the root lattice of $E_7$. 
\end{lemma}

Next, we discuss imaginary integral octonions with a given norm.  
Let $Q_7$ be the root lattice of $E_7$, and consider
$$\theta_{Q_7}(z)=\sum_{\alpha\in Q_7} q^{\frac 12 N(\alpha)}=1+\sum_{n=1}^\infty 
a_n(Q_7) q^n,
$$
where $q=e^{2\pi i z}$. Given a positive integer $n$, let 
$$N_{{\rm ioc}}(n)=\#\{ \text{$w\in\frak o'$ : $N(w)=n$}\}.$$
Then $a_n(Q_7)=N_{{\rm ioc}}(n)$. Now $\theta_{Q_7}\in M_{\frac 72}(\Gamma_0(4))$.
Let us review relevant facts on $M_{\frac 72}(\Gamma_0(4))$ from \cite{C}: dim $M_{\frac 72}(\Gamma_0(4))=2$. A basis is given by $\{\theta^7, \theta^3 F_2\}$, where
\begin{eqnarray*}
&& \theta(z)=\sum_{n=-\infty}^\infty q^{n^2}=1+2q+2q^4+\cdots,\\
&& F_2(z)=\frac {\eta(4z)^8}{\eta(2z)^4}=\sum_{n\geq 1\atop \text{$n$ odd}} \sigma_1(n)q^n=q+4q^3+6q^5+\cdots
\end{eqnarray*}
where $\eta$ is the usual eta function. 
By comparing the first two coefficients, we have 
\begin{equation}\label{rel}
\theta_{Q_7}(z)=\theta(z)^7+112 \theta(z)^3 F_2(z).
\end{equation}

By computing the right hand side of (\ref{rel}), 
we have

\begin{prop} $N_{{\rm ioc}}(1)=126, N_{{\rm ioc}}(2)=756, N_{{\rm ioc}}(3)=2072, N_{{\rm ioc}}(4)=4158, N_{{\rm ioc}}(5)=7560$, and $N_{{\rm ioc}}(6)=11592$. 
\end{prop}

We also need the following:

\begin{prop}\label{1/2} $\#\{ \text{$x$ imaginary octonion} : \frac 12+x\in \frak o, \, N(\frac 12+x)=1\}=56$. 
\end{prop}

\begin{proof} In the notation of \cite[p.574]{Co}, $\frac 12+x$ should be of the form $l_1\pm l_i$ or $l_2\pm l_i$, $i=3,...,8$ (total 24), and $\frac 12(l_1+l_2\pm l_3\pm\cdots\pm l_8)$ (odd number of minus signs) (total 32). 
\end{proof}

For the proof of the following propositions, we use Corollary \ref{Trace} and count the number of triples $(x,y,z)\in (\frak o')^3$ such that $Tr((\bar xy)\bar z)=2$ in Proposition \ref{1/2-3}, for example, and 
use Mathematica (version 12.1).

\begin{prop}\label{1/2-1} The number of pairs $(x,y)$ of imaginary octonions such that $\frac 12+x, \frac 12+y\in \frak o$, and $(\frac 12+\bar x)(\frac 12+y)=\frac 12+z$, where $z$ is an imaginary octonion, and $N(\frac 12+x)=1, N(\frac 12+y)=1$, is 1512. 
\end{prop}

\begin{prop}\label{1/2-2} The number of pairs $(x,y)$ of imaginary octonions such that $\frac 12+x, y\in \frak o$, and $(\frac 12+\bar x)y$ is an  imaginary octonion, and $N(\frac 12+x)=1, N(y)=1$, is 4032. 
\end{prop}

\begin{prop}\label{1/2-3} The number of pairs $(x,y)$ such that $x, y\in \frak o'$ and $\bar x y\in \frak o'$, and $N(x)=N(y)=1$, is 7560. 
\end{prop}

These results will be used in Section 6.

\section{Explicit formulas for Fourier coefficients}
In this section, we recall explicit forms of Fourier coefficients 
of Eisenstein series and Ikeda type lifts from \cite{Kim, KY}.

\subsection{The Fourier coefficients of Eisenstein series}
Let $E_{2k}$ be the Eisenstein series of weight $2k$ in \cite{Ba, Kim}:
$$E_{2k}(Z)=
\sum_{T\in \frak J(\Z)_{\geq 0}} \widetilde A_{2k}(T) e^{2\pi i(T,Z)},\ Z\in \frak T,
$$
where $\widetilde A_{2k}(T)=C_{2k}A_{2k}(T)$, and 
for $T\in\frak J(\Z)_{\ge 0}$ (rank 3), $A_{2k}(T)=\det(T)^{\frac{2k-9}{2}} \ds\prod_{p|\det(T)} \widetilde{f}_T^p(p^{\frac {2k-9}2})$, and  
$C_{2k}=-8\ds\prod_{j=0}^2 \frac {2k-4j}{B_{2k-4j}}$. Here $f_T^p(X)$ is a polynomial with integer coefficients and degree $d_p=v_p(\det(T))$ and constant term 1. Then we put $\widetilde f_T^p(T):=X^{d_p}f_T^p(X^{-2})$.
[Note that in \cite{Ka}, the Bernoulli numbers $B_g$ are defined to be $(-1)^{\frac g2-1}\times$ (the usual Bernoulli numbers).]

Here $\widetilde A_{k}(O)=1$, and if $T$ has rank 1, $\widetilde A_{2k}(T)=C_{2k}^{(1)}A_{2k}(T)$, $A_{2k}(T)=\Delta(T)^{2k-1}\ds\prod_{p| \Delta(T)} f_T^p(p^{1-k})$ and $C_{2k}^{(1)}=-\ds\frac {4k}{B_{2k}}$, where $\Delta(T)$ is defined in \cite{Ka}. [$T$ is equivalent over $\Bbb Z$ to $\begin{pmatrix} T_1&0\\0&0\end{pmatrix}$, where $\det(T_1)\ne 0$. Then $\Delta(T)=\det(T_1)$.] When $T$ has rank 1, $f_T^p(X)=\ds\sum_{i=0}^{d_p} X^i$. Hence $A_{2k}(T)=\sigma_{2k-1}(\Delta(T))$.

If $T$ has rank 2,
\begin{eqnarray*}
&& \widetilde A_{2k}(T)=C_{2k}^{(2)}A_{2k}(T),\quad C_{2k}^{(2)}= \frac {8k(2k-4)}{B_{2k}B_{2k-4}},\\
&& A_{2k}(T)=\Delta(T)^{2k-5} \prod_{p| \Delta(T)}  f_T^p(p^{5-2k})=\sum_{d| \epsilon(T)} d^{2k-1} \sigma_{2k-5}(\Delta(T\times T)/d^2),
\end{eqnarray*}
where $T\times T$ is the Jordan algebra product defined in (\ref{JAP}) and 
$\epsilon(T)$ stands for the content of $T$.

\begin{remark} There are typos in \cite{Kim} and \cite{KK}. There $C_{2k}$ should be 
$$C_{2k}=-\frac {8 (2k)(2k-4)(2k-8)}{B_{2k}B_{2k-4}B_{2k-8}}.
$$
and $C_{2k}^{(2)}=\frac {4 (2k)(2k-4)}{B_{2k}B_{2k-4}}.$
It comes from the fact that $vol(\frak C/\frak o)=2^{-4}$. 
\end{remark}

\subsection{Ikeda type lift on $\frak T$}

Suppose $T$ has rank 3. Recall the formula for $f_T^p(X)$ from \cite{Ka}.
Let $T=\diag(t_1,t_2,t_3)$ and let $\tau_p(i)={\rm ord}_p t_i$. For simplicity, let $\tau(i)=\tau_p(i)$. 
We may assume that $\tau(1)\leq \tau(2)\leq \tau(3)$. 
Let $d_p=\tau(1)+\tau(2)+\tau(3)$.

If $\tau(1)=0$,
then
$$f_T^p(X)=\sum_{l=0}^{\tau(2)} (p^4 X)^l \frac {1-X^{\tau(3)+\tau(2)+1-2l}}{1-X}=\sum_{l=0}^{\tau(2)} (p^4 X)^l 
(1+X+\cdots+X^{\tau(3)+\tau(2)-2l}).
$$
We can see easily that

\begin{eqnarray*}
&& \widetilde{f}_T^p(X)=(X^{d_p}+X^{d_p-2}+\cdots+X^{-d_p})+p^4 (X^{d_p-2}+X^{d_p-4}+\cdots+X^{-d_p+2})+\cdots \\
&& \phantom{xxxxxxxxxxx} + p^{4\tau(2)} (X^{d_p-2\tau(2)}+X^{d_p-2\tau(2)-2}+\cdots+X^{-d_p+2\tau(2)}).
\end{eqnarray*}

Hence
\begin{eqnarray}\label{formula1}
&& (p^{\frac {2k-9}2})^{d_p}\widetilde{f}_T^p(\alpha_p)=a_f(p^{d_p})+p^{2k-5}a_f(p^{d_p-2})+\cdots + p^{\tau(2)(2k-5)}a_f(p^{d_p-2\tau(2)}) \nonumber \\
&& \phantom{xxxxxxxxxxxx} =\sum_{i=0}^{\tau(2)} (p^i)^{2k-5} a_f(p^{d_p-2i}).
\end{eqnarray}
 
For any positive integer $2k\ge 20$, let $f(\tau)=\ds\sum_{n=1}^\infty a_f(n)q^n$ be a Hecke eigenform of
weight $2k-8$ with respect to $\SL_2(\Bbb Z)$, and let 
\begin{equation}\label{Ikeda-lift}
F_f(Z)=\sum_{T\in \frak J(\Z)_{>0}}A_{F_f}(T) e^{2\pi i(T,Z)},\quad
A_{F_f}(T)=\det(T)^{\frac{2k-9}{2}} \prod_{p|\det(T)} \widetilde{f}_T^p(\alpha_p) ,\ Z\in \frak T_2,
\end{equation}
be the Ikeda type lift, constructed in \cite{KY}. Here $a_f(p)=p^{\frac {2k-9}2}(\alpha_p+\alpha_p^{-1})$, where $\{\alpha_p,\alpha_p^{-1}\}$ are Satake parameters. Then, we have

\begin{prop}\label{Ikeda} If $\det(T)$ is square free, $A_{F_f}(T)=a_f(\det(T))$. 
Also, if $T$ is of the form $\diag(1,1,n)$ for a positive integer $n$, then $A_{F_f}(T)=a_f(n)$. 
\end{prop}
\begin{proof} By assumption, we have $\tau(1)=\tau(2)=0$ and $\tau(3)=d_p=
{\rm ord}_p(\det(T))$. 
The claim follows from (\ref{formula1}). 
\end{proof}

Similarly, for Eisenstein series, we have
\begin{prop}\label{Eisen-rank3} If $\det(T)$ is square free, then $A_{2k}(T)=\sigma_{2k-9}(\det(T))$. 
If $T$ is of the form $\diag(1,1,n)$ for a positive integer $n$, then $A_{2k}(T)=\sigma_{2k-9}(n)$. 
\end{prop}

Next we consider another explicit formula.
If $T=\begin{pmatrix}
1&x&y\\
\bar x&1&z\\ 
\bar y&\bar z&m
\end{pmatrix}\in \frak J(\Bbb Z)_{>0}$, then 
$\det\begin{pmatrix}
1 & x \\
\bar x & 1 
\end{pmatrix}=1-N(x)
>0$ (see Section \ref{excep} for the criterion). Thus, $x=0$, and $T\times T=\begin{pmatrix} m-N(z)&-y\bar z&-y\\ c\bar b&m-N(y)& -z\\ -\bar y&-\bar z&1\end{pmatrix}$.
Hence $T$ and $T\times T$ are primitive. Here $\det(T)=m-N(y)-N(z)$. Therefore, for each prime $p$, $T$ is equivalent to $\diag(1,1,\det(T))$ over $\Bbb Z_p$.
Therefore, by Proposition \ref{Ikeda}, we have

\begin{prop}\label{Ikeda1} If $T=\begin{pmatrix} 1&x&y\\ \bar x&1&z\\ \bar y&\bar z&m\end{pmatrix}\in 
\frak J(\Bbb Z)_{> 0}$, then $A_{F_f}(T)=a_f(m-N(y)-N(z)).$
\end{prop}

\begin{prop}\label{Ikeda2} Let $p$ be a prime number. If $T=\diag(p,p,p)\in 
\frak J(\Bbb Z)_{> 0}$, then $A_{F_f}(T)=a_f(p)^3+p^{2k-9}(p^8+p^4-2)a_f(p).$
\end{prop}
\begin{proof}By the formula in \cite[p.201]{Ka}, we see that 
$$f^p_T(X)=X^3+(p^8+p^4+1)X^2+(p^8+p^4+1)X+1.$$
Since $d_p=3$, we have $\widetilde{f}^p_T(X)=
(X+X^{-1})^3+(p^8+p^4-2)(X+X^{-1})$. Thus, it yields 
$$A_{F_f}(T)=(p^{\frac{2k-9}{2}})^3\widetilde{f}^p_T(\alpha_p)
=a_f(p)^3+p^{2k-9}(p^8+p^4-2)a_f(p)$$
as desired.
\end{proof}

\begin{prop}\label{Ikeda3} Let $p$ be a prime number. If $T=\diag(1,p,p)\in 
\frak J(\Bbb Z)_{> 0}$, then $A_{F_f}(T)=a_f(p^2)+p^{2k-5}.$
\end{prop}
\begin{proof} By the formula (\ref{formula1}), 
$\widetilde{f}^p_T(X)=
X^2+1+X^{-2}+p^4$. Thus, it yields 
$$A_{F_f}(T)=(p^{\frac{2k-9}{2}})^2\widetilde{f}^p_T(\alpha_p)
=a_f(p^2)+p^{2k-5}.$$
\end{proof}

\section{Siegel modular forms of $Sp_6$ of level one}

We record here Siegel modular forms of $Sp_6$ of level one, computed by Yuen-Poor-Shurman \cite{YPS}.
Consider weight 20 case: dim$M_{20}(\Gamma_3)=11$ and dim$S_{20}(\Gamma_3)=6$. In their notations, 
\begin{enumerate}
  \item $f_1=E_{20}^{(3)}$: Siegel Eisenstein series;
  \item $f_2$ is Klingen Eisenstein series formed from the weight 20 elliptic cusp form;
  \item $f_3, f_4$ are Klingen Eisenstein formed from Saito-Kurokawa lifts; Since dim$S_{38}(\SL_2(\Bbb Z))=2$, the two forms give rise to two Saito-Kurokawa lift of weight 20 and degree 2;
  \item $f_5$ is Klingen Eisenstein series formed from the unique degree 2 genuine form;
  \item $f_6,f_7,f_8$ are Miyawaki lifts of type I. They are given as follows;
For $f\in S_{2k-4}(\SL_2(\Bbb Z))$ and $g\in S_{k}(\SL_2(\Bbb Z))$, we obtain $F_{f,g}\in S_k(\Sp_6(\Bbb Z))$.
Since dim$S_{36}(\SL_2(\Bbb Z))=3$ and dim$S_{20}(\SL_2(\Bbb Z))=1$, they are given by $F_{f,g}$, where $f\in S_{36}(\SL_2(\Bbb Z))$ and $g\in S_{20}(\SL_2(\Bbb Z))$ in the notation of \cite{Ik}.
\item $f_9, f_{10}$ are conjectural Miyawaki lifts of type II; They are conjecturally given as follows;
For $f\in S_{2k-2}(\SL_2(\Bbb Z))$ and $g\in S_{k-2}(\SL_2(\Bbb Z))$, we obtain $F_{f,g}\in S_k(\Sp_6(\Bbb Z))$.
Since dim$S_{38}(\SL_2(\Bbb Z))=2$ and dim$S_{18}(\SL_2(\Bbb Z))=1$, they are given by $F_{f,g}$, where $f\in S_{38}(\SL_2(\Bbb Z))$ and $g\in S_{18}(\SL_2(\Bbb Z))$ in the notation of \cite{Ik}.
\item $f_{11}$ is the genuine form which is not a lift.
\end{enumerate}


\section{Restriction of Modular forms on $E_{7,3}$ to $Sp_6$}\label{RMF} 
Let $F$ be a modular form on $\frak T$, and consider $F|_{Sp_6}$. 
We write $Z=Z_1+Z_2$ for $Z\in \frak T$ and $T=T_1+T_2$ for $T\in \frak J(\Z)_{\ge 0}$ 
as in Section \ref{excep} so that $Z_1\in \mathbb{H}_3$ and $T_1\in {\rm Sym}^3(\Z)_{\ge 0}$. 
The Fourier coefficients of the restriction $F|_{Sp_6}$ is described in terms of those of $F$ 
as follows:
\begin{lemma}\label{someformula}It holds that 
$$F|_{Sp_6}(Z_1)=F(Z_1)=\sum_{T\in \frak J(\Bbb Z)_{\ge 0}} A_F(T) e^{2\pi i (T_1,Z_1)}=\sum_{S\in 
{\rm Sym}^3(\Bbb Z)_{\ge 0}} 
\left(\sum_{T\in\frak J(\Bbb Z)_{\ge 0}\atop T_1=S} A_F(T)\right) e^{2\pi i {\rm Tr}(S Z_1)}.
$$
\end{lemma}
\begin{proof}The claim easily follows from $(T_2,Z_1)=0$.
\end{proof}

For $S\in {\rm Sym}^3(\Bbb Z)_{\ge 0}$, let $A_{F|_{Sp_6}}(S):=\sum_{T\in\frak J(\Bbb Z)_{\ge 0}\atop T_1=S} A_F(T)$ so that
$$F|_{Sp_6}(Z_1)=\sum_{S\in {\rm Sym}^3(\Bbb Z)_{\ge 0}} A_{F|_{Sp_6}}(S) e^{2\pi i {\rm Tr}(S Z_1)}.
$$

For modular forms on $\frak T$ of weight $k$ with level one, we give a strategy to write their restriction to $Sp_6$ as a linear combination of 
a basis of $M_k(\Sp_6(\Z))$ as follows:
\begin{enumerate}
\item  Give formulas for $A_{F|_{Sp_6}}(S)$ for particular $S\in {\rm Sym}^3(\Z)_{\ge 0}$. 
(see Propositions 4.2 to 4.6);
\item By using an explicit basis $f_1,\ldots,f_{d_k}$ of $M_k(\Sp_6(\Z))$ due to Yuen-Poor-Shurman, we test and repeat if $A:=(A_{f_i}(S_j))_{1\le i,j\le d_k}$ is invertible for some 
$S_1,\ldots,S_{d_k}\in {\rm Sym}^3(\Z)_{\ge 0}$ of which we can apply the above formulas;
\item If the second step fails, we add new elements 
$H_1,\ldots,H_a\in  {\rm Sym}^3(\Z)_{\ge 0}$ so that the matrix obtained by appending 
$d_k\times a$-matrix $(A_{f_i}(H_j))_{1\le i\le d_k\atop 1\le j\le a}$ to the matrix $A$ is of rank $d_k$.
Then, we compute $A_{F|_{Sp_6}}(H_j)$ for $1\le j\le a$. 
In the course of the computation, for each $1\le j\le a$, we need to list up all $T\in \frak J(\Z)_{\ge 0}$ 
such that $T_1=H_j$  and this part is tedious (see Section \ref{Tdiag222}).   
\item Even if we could successfully have a linear combination for specific $\{S_j\}$ 
(with $\{H_j\}$ if necessary), we can check if the resulting one is actually true by 
choosing  another matrices of ${\rm Sym}^3(\Z)_{\ge 0}$ to confirm. 
We have carried out this procedure for each examples we obtained. 
\end{enumerate}

\begin{lemma}\label{vanishing}
Assume $F$ is a cusp form. Then, for 
$S=\begin{pmatrix}
1&\frac 12& a\\ 
\frac 12&1&b\\
 a&b&c
 \end{pmatrix}\in {\rm Sym}^3(\Bbb Z)_{\ge 0}$, $A_{F|_{Sp_6}}(S)=0$.
\end{lemma}
\begin{proof}
Let $T=\begin{pmatrix}
1&\frac 12+x & \ast\\ 
\frac 12+\bar{x}&1&\ast\\
 \ast&\ast&\ast
 \end{pmatrix}$ such that $T_1=S$ where $x$ is an imaginary octonion in $\frac{1}{2}\frak o$ such that  $\frac{1}{2}+x\in \frak o$. Since $F$ is a cusp form, 
 we may assume $T\in \frak J(\Z)_{>0}$.  
By applying the criterion in Section \ref{excep}, we have $1>N(\frac{1}{2}+x)$ and it yields 
$\frac{1}{2}+x=0$. Since $x$ is an imaginary octonion, it is impossible. 
Thus, there is no element $T\in \frak J(\Bbb Z)_{> 0}$ such that $T_1=S$. Hence for such $S$, $A_{F|_{Sp_6}}(S)=0$. 
\end{proof}
\begin{remark}\label{v-rm}Put $x=-\alpha_1-\alpha_5=
\frac{1}{2}(e_0+e_1+e_4-e_5)\in \frak o$. 
Let $T=\begin{pmatrix}
1& x & 0\\ 
x&1& 0\\
 0&0& 1
 \end{pmatrix}$. Then, $T\in \frak J(\Z)_{\ge 0}\setminus \frak J(\Z)_{>0}$, since $\det(\begin{pmatrix}
1& x \\ 
x&1  
\end{pmatrix}
)=1-N(x)=0$, but $T_1=\begin{pmatrix}
1& \frac{1}{2} & 0\\ 
\frac{1}{2}&1& 0\\
 0&0& 1
 \end{pmatrix}\in  {\rm Sym}^3(\Bbb Z)_{>0}$. Hence the claim of Lemma \ref{vanishing} 
is not true in general for non-cusp forms.  
\end{remark}

\begin{prop}\label{Da} Let $F=F_f$ be the Ikeda type lift defined in $($\ref{Ikeda-lift}$)$. 
For $D_a=\diag(1,1,a)$ with $a\in \Z_{>0}$, 
$$A_{F|_{Sp_6}}(D_a)=\sum_{n=1}^a a_f(n) \#\{(y,z)\in \frak o'\times \frak o'\ |\ 
N(y)+N(z)=a-n,\ N(y)<a,\ N(z)<a\}.
$$
\end{prop}
\begin{proof} Consider $T\in \frak J(\Bbb Z)_{>0}$ such that $T_1=\diag(1,1,a)$. Then $T=\begin{pmatrix}
1& 0 & y\\ 
0&1& z\\
-y&-z& a
 \end{pmatrix}$, where $y,z\in\frak o'$ such that $N(y)<a, N(z)<a$. Also $\det(T)=a-N(y)-N(z)$. The assertion follows from Proposition \ref{Ikeda1}.
\end{proof}

\begin{prop} Let $F=F_f$ be the Ikeda type lift defined in $($\ref{Ikeda-lift}$)$. 
Let $G=\begin{pmatrix} 1&\frac 12& 0\\ \frac 12&2&0\\ 0&0&2\end{pmatrix}$.
Then $A_{F|_{Sp_6}}(G)=56a_f(2)+ 11088$.
\end{prop}

\begin{proof}
 Let $T=\begin{pmatrix} 1&\frac 12+x& y\\ \frac 12+\bar x&2&z\\ \bar y&\bar z&2\end{pmatrix}\in \mathcal J(\Bbb Z)_{>0}$, where $y,z$ are imaginary integral octonions, and $x$ is an imaginary octonion such that $\frac 12+x\in\frak o$.
Since $2-N(\frac 12+x)>0$ and $\frac 12+x\ne 0$, $N(\frac 12+x)=1$. Then $T$ is equivalent over $\Bbb Z$ to $\begin{pmatrix} 1&0&\bar y\\ 0&1&z-(\frac 12+\bar x)y\\ \bar y&\bar z-\bar y(\frac 12+x)&2\end{pmatrix}$, and it is equivalent to $\begin{pmatrix} 1&0&0\\ 0&1&0\\ 0&0&2-N(y)-N(z-(\frac 12+\bar x)y))\end{pmatrix}$.

If $N(y)=0$, $y=0$ and $N(z)=0$ or 1. If $z=0$, $A_F(T)=a_f(2)$. 

If $N(z)=1$, $A_F(T)=1$. There are $56\times 126$ pairs $(x,z)$.

If $N(y)=1$, since $2-N(y)-N(z-(\frac 12+\bar x)y)>0$, $N(z-(\frac 12+\bar x)y)=0$. Hence $z=(\frac 12+\bar x)y$. Therefore, in this case, 
$A_F(T)=1$. By Proposition \ref{1/2-2}, there are 4032 such pairs $(x,y)$.
\end{proof}

The following proposition is proved by using the data in Appendix A. 
\begin{prop}\label{coef222} Let $F=F_f$ be the Ikeda type lift defined in 
$($\ref{Ikeda-lift}$)$. 
Put $H=\diag(2,2,2)\in {\rm Sym}^3(\Z)_{>0}$. 
Then \begin{eqnarray*}
A_{F|_{Sp_6}}(H)&=&a_f(2)^3+ 270\cdot 2^{2k-9} a_f(2)+2268 (a_f(4)+2^{2k-5})
+378 a_f(6)+ 55188 a_f(4)\\
&&+459648 a_f(3)+ 3752952 a_f(2)+18192384.
\end{eqnarray*}
\end{prop}

\begin{proof} All $T\in \frak J(\Z)_{>0}$ such that $T_1=H$ are classified in Table 1 in Appendix A. 
Applying the formula (\ref{formula1}), and 
Propositions \ref{Ikeda}, \ref{Ikeda1}, \ref{Ikeda2}, \ref{Ikeda3} for each $T$ with $d(T)=(2,4,8)$, $(1,2,4)$, and $(1,1,n),\ 
n\in\{1,2,3,4,6\}$, respectively, the result follows. 
\end{proof}

\subsection{Weight 20 Ikeda type lift}
Now we consider the Ikeda type lift $F=F_\Delta$ of weight 20 from the Ramanujan $\Delta$ function, which is the weight 12 form with respect to $\SL_2(\Bbb Z)$: 
$$\Delta(z)=\eta(z)^{24}=q-24q^2+252q^3-1472q^4+4830q^5-6048q^6+\cdots. 
$$
Since $F|_{Sp_6}\in S_{20}(\Sp_6(\Bbb Z))$,
\begin{equation}\label{linear}
F|_{Sp_6}=c_1 f_6+c_2 f_7+c_3 f_8+c_4 f_9+c_5 f_{10}+c_6 f_{11}.
\end{equation}
Put $A(S)=A_{F|_{Sp_6}}(S),\ S\in {\rm Sym}^3(\Z)_{\ge 0}$ for simplicity. 

We plug in $S_1,W,D_1,D_2,G,H$, and obtain a rank 6 matrix from \cite{YPS}, where 
$$S_1=\begin{pmatrix} 1&\frac 12& \frac 12\\ \frac 12&1&\frac 12\\ \frac 12&\frac 12&1\end{pmatrix},\ 
W=\begin{pmatrix} 1&\frac 12& 0\\ \frac 12&1&0\\ 0&0&1\end{pmatrix},\ 
G=\begin{pmatrix} 1&\frac 12& 0\\ \frac 12&2&0\\ 0&0&2\end{pmatrix},\ 
H=\begin{pmatrix} 2&0& 0\\ 0&2&0\\ 0&0&2\end{pmatrix}.$$ 
Then $A(S_1)=A(W)=0$, and applying Proposition \ref{Da},  $A(D_1)=1$, $A(D_2)=228$, and $A(G)=9744$. Further, applying Proposition \ref{coef222} with $k=10$, we have 
$A(H)=18124416$.
By solving the linear system associated to the above matrix of rank 6,  we have

\begin{eqnarray*}
&&c_1=\frac {\beta-\gamma(b_2+b_3)+\delta b_2b_3}{\alpha(b_1-b_2)(b_1-b_3)},\quad c_2=\frac {-\beta+\gamma(b_1+b_3)-\delta b_1b_3}{\alpha(b_1-b_2)(b_2-b_3)},\quad c_3=\frac {\beta-\gamma (b_1+ b_2)+\delta b_1b_2}{\alpha(b_1-b_3)(b_2-b_3)} ,\\
&&c_4=\frac {\mu+ 9263 a}{\eta a},\quad 
c_5=\frac {-\mu+ 9263 a}{\eta a},\quad c_6=-\frac 1{\kappa},
\end{eqnarray*}
where 
\begin{eqnarray*}
&& \alpha=2^{50}\cdot 3^{32}\cdot 5^9\cdot 7^7\cdot 11^6\cdot 13^3\cdot 17^2\cdot 19\cdot 23\cdot 29\cdot 31\cdot 71\cdot 83\cdot 157\cdot 313\cdot 439367\cdot 249789005569\\
&& \gamma=5\cdot 11\cdot 19\cdot 59\cdot 2633441,\quad \delta=7 \cdot 337\cdot 2099,\\
&& \beta=5\cdot 11^2\cdot 109\cdot 77263111477,\quad \mu=1447\cdot 203953,\\
&& \eta=2^{48}\cdot 3^{22}\cdot 5^{14}\cdot 7^{9}\cdot 11^4\cdot 13^3\cdot 17\cdot 29\cdot 31\cdot 67\cdot 83^2\cdot 1699\cdot 2069\cdot 78803, \\
&& \kappa=2^{40}\cdot 3^{18}\cdot 5^{10}\cdot 7^2\cdot 11^2\cdot 13\cdot 17\cdot 29^2\cdot 31\cdot 83\cdot 157\cdot 271\cdot 1009.
\end{eqnarray*}
where $a=\sqrt{63737521}$, and
$b_1,b_2,b_3$ are roots of $x^3-54971x^2+893191104x-4382089113600=0$, approximately, 
$b_1\approx 9504, b_2\approx 15267, b_3\approx 30199$.

One can check for consistency by computing both sides of (\ref{linear}) for $D_3,D_4$.

\subsection{Restriction of Eisenstein series to $Sp_6$}

The restriction of Eisenstein series is more complicated: 
\begin{prop}\label{Da-Eis} For $D_a=\diag(1,1,a)$ with $a\in \Z_{>0}$, 
\begin{eqnarray*}
&& A_{E_{2k}|_{Sp_6}}(D_a)=C_{2k}\sum_{n=1}^a \sigma_{2k-9}(n) \#\{(y,z)\in \frak o'\times \frak o'\ |\ N(y)+N(z)=a-n,\ N(y)<a,\ N(z)<a\} \\ && +C_{2k}^{(2)}\#\{(y,z)\in \frak o'\times \frak o'\ |\ N(y)+N(z)=a\} \\ && +C_{2k}^{(2)} \sum_{n=1}^a \sigma_{2k-5}(n) \#\{(x,y,z)\in \frak o'\times \frak o'\times \frak o'\ |\ N(x)=1,\, a-N(y)=n,\, z=\bar x y \} \\ && +C_{2k}^{(1)}\#\{(x,y,z)\in \frak o'\times \frak o'\times \frak o'\ |\ N(x)=1,\, N(y)=a,\, z=\bar x y\}. \end{eqnarray*}
\end{prop}
\begin{proof} Consider $T\in \frak J(\Bbb Z)_{\geq 0}$ such that $T_1=\diag(1,1,a)$. Then $T=\begin{pmatrix} 1& x & y\\ \bar x&1& z\\ \bar y&\bar z& a \end{pmatrix}$, where $x,y,z\in\frak o'$. Then $N(x)\leq 1$. If $x=0$, as in the proof of Proposition \ref{Ikeda1}, $\det(T)=a-N(y)-N(z)$, and $T$ is equivalent to $\diag(1,1,\det(T))$ over $\Bbb Z_p$. If $\det(T)\ne 0$, then $A_{2k}(T)=\sigma_{2k-9}(\det(T))$. If $\det(T)=0$, $A_{2k}(T)=1$. 

If $N(x)=1$, $m_{-x e_{12}}\cdot T=\begin{pmatrix} 1&0 & y\\ 0&0& z-\bar x y\\ \bar y&\bar z-\bar y x&a \end{pmatrix}$, where $m_{-x e_{12}}$ is as in \cite[p. 228]{KY}. Since $m_{-x e_{12}}\cdot T\in \frak J(\Bbb Z)_{\geq 0}$, $z-\bar x y=0$. Hence $T$ is equivalent to $\diag(1,a-N(y),0)$ over $\Bbb Z_p$. If $N(y)<a$, then $A_{2k}(T)=\sigma_{2k-5}(a-N(y))$. If $N(y)=a$, then $T$ is equivalent to $\diag(1,0,0)$ over $\Bbb Z_p$.
\end{proof}

Here the difficulty is to compute the number of pairs $(x,y)\in\frak o'\times\frak o'$ such that $\bar xy\in\frak o'$. However, if $a=1$, we can use Proposition \ref{1/2-3}.

 \begin{prop} $A_{E_{2k}|_{Sp_6}}(\diag(1,1,1))=C_{2k} + 378 C_{2k}^{(2)} + 7560 C_{2k}^{(1)}$.
\end{prop}

 \begin{proof} By the above proposition, 
 \begin{eqnarray*}
 A_{E_{2k}|_{Sp_6}}(D_1)=C_{2k} + 126\times 2 C_{2k}^{(2)} +126 C_{2k}^{(2)} +C_{2k}^{(1)}\#\{(x,y)\in \frak o'\times \frak o'\ |\ N(x)=1,\, N(y)=1,\, \bar x y\in \frak o'\}.\end{eqnarray*}

By Proposition \ref{1/2-3},
$$\#\{(x,y)\in \frak o'\times \frak o'\ |\ N(x)=1,\, N(y)=1,\, \bar x y\in \frak o'\}=7560.
$$
\end{proof}

Let 
$$u_2=\diag(1,0,0),\quad u_4=\diag(1,1,0),\quad u_6=\begin{pmatrix} 1&\frac 12&0\\ \frac 12&1&0\\ 0&0&0\end{pmatrix},\quad W=\begin{pmatrix} 1&\frac 12&0\\ \frac 12&1&0\\ 0&0&1\end{pmatrix}.
$$  
For simplicity, let $A(S)=A_{E_{2k}|_{Sp_6}}(S)$.

\begin{lemma} $A(u_2)=C_{2k}^{(1)}$, $A(u_4)=C_{2k}^{(2)}+126 C_{2k}^{(1)}$, $A(u_6)=56 C_{2k}^{(1)}$, and $A(W)=56C_{2k}^{(2)}+ 4032 C_{2k}^{(1)}$. 
\end{lemma}
\begin{proof} If $T=\begin{pmatrix} 1&x&y\\ \bar x&0&z\\ \bar y&\bar z&0\end{pmatrix}\in \frak J(\Bbb Z)_{\geq 0}$ such that $T_1=u_2$, then $x=y=z=0$. Hence $A(u_2)=C_{2k}^{(1)}$.

 Suppose $T=\begin{pmatrix} 1&x&y\\ \bar x&1&z\\ \bar y&\bar z&0\end{pmatrix}$ such that $T_1=u_4$ so that $x,y,z$ are imaginary integral octonions. Since $T\in \frak J(\Bbb Z)_{\geq 0}$, $y=z=0$, and $N(x)\leq 1$. If $x=0$, then $T=\diag(1,1,0)$. In that case, $A(\diag(1,1,0))=C_{2k}^{(2)}$. If $N(x)=1$, then $T$ is equivalent to $\diag(1,0,0)$ over $\Bbb Z_p$. Hence $A(u_4)=C_{2k}^{(2)}+126 A(u_2)$.

Suppose $T=\begin{pmatrix} 1&\frac 12+x&y\\ \frac 12+\bar x&1&z\\ \bar y&\bar z&0\end{pmatrix}\in \frak J(\Bbb Z)_{\geq 0}$ such that $T_1=u_6$ so that $y,z$ are imaginary integral octonions and $x$ is an imaginary octonion such that $\frac 12+x\in\frak o$. Since $T\in \frak J(\Bbb Z)_{\geq 0}$, $y=z=0$, and $N(\frac 12+x)\leq 1$. Since $N(\frac 12+x)\ne 0$, $N(\frac 12+x)=1$. Then $\begin{pmatrix} 1&\frac 12+x&0\\ \frac 12+\bar x&1&0\\ 0&0&0\end{pmatrix}$ is equivalent to $\diag(1,0,0)$ over $\Bbb Z_p$. Now use Proposition \ref{1/2}.

 Suppose $T=\begin{pmatrix} 1&\frac 12+x&y\\ \frac 12+\bar x&1&z\\ \bar y&\bar z&1\end{pmatrix}\in \frak J(\Bbb Z)_{\geq 0}$ such that $T_{1}=W$ so that $y,z$ are imaginary integral octonions and $x$ is an imaginary octonion such that $\frac 12+x\in\frak o$. As above, $N(\frac 12+x)=1$, and $z-(\frac 12+\bar x)y=0$. Then $T$ is equivalent to $\diag(1,1-N(y),0)$ over $\Bbb Z_p$. If $y=0$, $z=0$, and $A(T)=C_{2k}^{(2)}$.
If $N(y)=1$, $T$ is equivalent to $\diag(1,0,0)$ over $\Bbb Z_p$, and by Proposition \ref{1/2-2},
the number of pairs $(x,y)$ such that $x$ is an imaginary octonion, and $\frac 12+x\in\frak o, N(\frac 12+x)=1$, and $y\in \frak o'$, and
$(\frac 12+\bar x)y\in \frak o'$, is 4032.
\end{proof}

\begin{prop} Let $S_1=\begin{pmatrix} 1&\frac 12& \frac 12\\ \frac 12&1&\frac 12\\ \frac 12&\frac 12& 1 \end{pmatrix}$.
Then $A_{E_{2k}|_{Sp_6}}(S_1)=1512 C_{2k}^{(1)}.$
 \end{prop}
\begin{proof}
Consider
 $T=\begin{pmatrix} 1&\frac 12+x & \frac 12+y\\ \frac 12+\bar{x}&1&\frac 12+z\\ \frac 12+\bar y&\frac 12+\bar z&1\end{pmatrix}\in\frak J(\Bbb Z)_{\geq 0}$, where 
$x,y,z$ are imaginary octonions such that $\frac 12+x,\frac 12+y, \frac 12+z\in\frak o$. Then $N(\frac 12+x)\leq 1$. Since $N(\frac 12+x)\ne 0$, $N(\frac 12+x)=1$. Also $N(\frac 12+y)=1$, and $\frac 12+z=(\frac 12+\bar x)(\frac 12+y)$.
Then $T$ is equivalent to $\diag(1,0,0)$ over $\Bbb Z_p$. By Proposition \ref{1/2-1}, the number of such pairs $(x,y)$ is 1512.
\end{proof}

Now dim $M_{2k}(\Sp_6(\Z))=1$ for $2k=4,6,8$, generated by the Siegel Eisenstein series \cite{Ts}. Hence
for $2k=4,8$, $E_{2k}|_{Sp_6}=E_{2k}^{(3)}$ if we normalize the constant term of $E_{2k}^{(3)}$ to be one.

Recall that there are no holomorphic Eisenstein series of weight 6 and 10 on the exceptional domain \cite{Kim}.

\subsubsection{Weight 12}

 We have dim$M_{12}(\Gamma_3)=4$; $f_1$ is the Siegel Eisenstein series; $f_2,f_3$ are Klingen Eisenstein series; $f_4$ is the Miyawaki lift of type I. We plug in $O,u_2,u_6$ and $D_1$. Note that $A(D_1)=\frac {2^7\cdot 3^4\cdot 5^3\cdot 7\cdot 13^3}{691}$. So 
\begin{equation} \label{E12}
E_{12}|_{Sp_6}=c_1 f_1+c_2f_2+c_3f_3+c_4f_4, 
\end{equation}
 where $c_1=\frac {2^{13}\cdot 3^7\cdot 5^3\cdot 7^2\cdot 11\cdot 13\cdot 19\cdot 23}{131\cdot 283\cdot 593\cdot 617\cdot 691\cdot 43867}$, $c_2=0$, and $c_3=\frac {3\cdot 7\cdot 11\cdot 13\cdot 557}{5^2\cdot 131\cdot 593\cdot 691\cdot 43867}$, and $c_4=\frac {13}{2^6\cdot 3^5\cdot 5^2\cdot 17\cdot 691}$.

One can check for consistency by computing both sides of (\ref{E12}) for $D_2,W$.

\subsubsection{Weight 14}

 We have dim$M_{14}(\Gamma_3)=3$; $f_1$ is the Siegel Eisenstein series; $f_2$ is Klingen Eisenstein series; $f_3$ is the conjectural Miyawaki lifts of type II. We plug in $O,u_6$, and $D_1$. Note that $A(D_1)=-979776$.
\begin{equation} \label{E14}
E_{14}|_{Sp_6}=c_1 f_1+c_2f_2+c_3f_3,
\end{equation}
 where $c_1=\frac {2^{13}\cdot 3^5\cdot 5\cdot 7\cdot 13\cdot 23}{103\cdot 131\cdot 593\cdot 657931\cdot 2294797}$, $c_2=-\frac {1121}{3\cdot 5^3\cdot 7\cdot 131\cdot 593\cdot 691\cdot 657931}$, and $c_3=\frac {11}{2^7\cdot 5^4\cdot 7^2\cdot 43\cdot 691}$.

One can check for consistency by computing both sides of (\ref{E14}) for $D_2,W$.

 \subsubsection{Weight 16}

 We have dim$M_{16}(\Gamma_3)=7$; $f_1$ is the Siegel Eisenstein series; $f_2,f_3,f_4$ are Klingen Eisenstein series; $f_5,f_6$ are the Miyawaki lifts of type I; $f_7$ is the genuine form which is not a lift. We plug in $O,u_2,u_6,u_4, W, S$ and $D_1$. Note that $A(D_1)=\frac {2^9\cdot 3^7\cdot 5^2\cdot 7^2\cdot 17\cdot 43}{691\cdot 3617}$, and $A(u_2)=C_{16}^{(1)}=\frac {16320}{3617}$. Then
\begin{equation}\label{E16}
E_{16}|_{Sp_6}=c_1 f_1+c_2f_2+c_3f_3+c_4f_4+c_5f_5+c_6f_6+c_7f_7,
\end{equation}
where 
$$c_1=\frac {2^{15}\cdot 3^5\cdot 5^2\cdot 7\cdot 11\cdot 17\cdot 29\cdot 31}{1721\cdot 3617\cdot 9349\cdot 362903\cdot 657931\cdot 1001259881},
\quad c_2=0,
$$
$$c_3=\frac {\alpha+\beta a}{\gamma a},\quad c_4=\frac {-\alpha+\beta a}{\gamma a},
$$
$$c_5=\frac {17(-750637+3971 b)}{\delta b},\quad c_6=\frac {17((750637+3971 b)}{\delta b},
$$ 
$$ c_7=-\frac {7}{2^{10}\cdot 3\cdot 5^5\cdot 11\cdot 13\cdot 23\cdot 107\cdot 691\cdot 3617},
$$
where 
\begin{eqnarray*}
&& \alpha=7213\cdot 3027811\cdot 95518669,\quad \beta=2\cdot 397\cdot 18013\cdot 370716527,\\
&& \gamma=3^7\cdot 5^7\cdot 7^2\cdot 11^2\cdot 13\cdot 23\cdot 37\cdot 97\cdot 691\cdot 1721\cdot 1889\cdot 3617\cdot 657931\cdot 883331\cdot 1001259881,\\ 
&& \delta=2^{17}\cdot 3^8\cdot 5^9\cdot 7^4\cdot 11^2\cdot 13\cdot 19\cdot 23\cdot 107\cdot 373\cdot 691\cdot 3617,\\
&& a=\sqrt{51349},\quad b=\sqrt{18209}.
\end{eqnarray*}

One can check for consistency by computing both sides of (\ref{E16}) for $D_2,D_3$.

\subsubsection{Weight 18}

We have dim$M_{18}(\Gamma_3)=8$; $f_1$ is the Siegel Eisenstein series; $f_2,f_3,f_4$ are Klingen Eisenstein series; $f_5,f_6$ are the Miyawaki lifts of type I; $f_7,f_8$ are Miyawaki lifts of type II. We plug $O,u_2,u_6,u_4,W,S,D_1,H$. Since it is cumbersome to compute $A(H)=A_{E_{2k}|_{Sp_6}}(H)$, we set $A(H)=d$. Then
\begin{equation}\label{E18}
E_{18}|_{Sp_6}=c_1 f_1+c_2f_2+c_3f_3+c_4f_4+c_5f_5+c_7f_7+c_8f_8,
\end{equation}
where 
\begin{eqnarray*} 
&& c_1=\frac {2^{15}\cdot 3^7\cdot 5\cdot 7^2\cdot 11\cdot 17\cdot 19\cdot 31}{37\cdot 683\cdot 1721\cdot 43867\cdot 305065927\cdot 1001259881\cdot 151628697551},\quad c_2=0,\\
&& c_3=\frac {323(\alpha+\beta a)}{\gamma a},\quad c_4=\frac {-323(-\alpha+\beta a)}{\gamma a},\\
&& c_5=\frac {-\delta+\eta b+\mu d+\nu bd}{\xi b},\quad c_6=\frac {\delta+\eta b-\mu d+\nu bd}{\xi b},\\
&& c_7=\frac {-\sigma+\kappa a-\epsilon d+ \phi ad}{\Sigma a},\quad c_8=\frac {\sigma+\kappa a+\epsilon d+ \phi ad}{\Sigma a},
\end{eqnarray*}
and $a=\sqrt{2356201}, b=\sqrt{18295489}$,
\begin{eqnarray*}
&&\alpha=7\cdot 3763811\cdot 20444077\cdot 123989669,\quad \beta=163169\cdot 376345229688521,\\
&&\gamma={\scriptstyle 2^{13}\cdot 3^6\cdot 5^9\cdot 7^2\cdot 13^2\cdot 29^2\cdot 103\cdot 1721\cdot 3617\cdot 43867\cdot 1001259881\cdot 151628697551\cdot 7947218274280511539,}\\
&&\delta={\scriptstyle 2^6\cdot 3^6\cdot 7\cdot 17\cdot 19\cdot 53\cdot 107\cdot 15131\cdot 19684440289},\quad \eta={\scriptstyle 2^6\cdot 3^6\cdot 7\cdot 19\cdot 23629\cdot 3314029\cdot 14628829,}\\
&& \mu=13\cdot 19\cdot 307\cdot 719\cdot 43867,\quad \nu=13\cdot 251\cdot 419\cdot 43867,\quad \epsilon=79\cdot 3617\cdot 43867\cdot 1949111,\\
&& \sigma={\scriptstyle 2^6\cdot 3^6\cdot 7\cdot 19\cdot 59\cdot 8592181585668944485139},\quad \kappa={\scriptstyle 2^6\cdot 3^6\cdot 7\cdot 13\cdot 19\cdot 166031\cdot 1890701\cdot 83097733,}\\
&& \phi={\scriptstyle 17\cdot 3617\cdot 6043\cdot 43867},\quad \Sigma={\scriptstyle 2^{49}\cdot 3^{27}\cdot 5^{14}\cdot 7^8\cdot 11^4\cdot 13\cdot 29\cdot 103\cdot 3617\cdot 15511\cdot 43867\cdot 12519167993.}
\end{eqnarray*}

One can check for consistency by computing both sides of (\ref{E18}) for $D_2,D_3$.

\section{Appendix A: All $T\in \frak J(\Z)_{>0}$ with $T_1=\diag(2,2,2)$}
\label{Tdiag222} 
Let us follow the notation of Section \ref{RMF}.
We compute all $T=(x_{ij})_{1,\le i,j\le 3}\in \frak J(\Z)_{>0}$ with $T_1=\diag(2,2,2)$ 
together with the following invariants. 
Let 
$$d_1={\rm gcd}(T):={\rm gcd}(x_{ij}\ |\ 1\le i,j\le 3),\ d_2={\rm gcd}(T\times T),\ 
d_3=\det(T)
$$
and put 
$$d(T):=(d_1,d_2,d_3).$$
Then, for each rational prime $p$, put 
$$\tau_p(T):=(\tau(1),\tau(2),\tau(3)),\ \tau(i):={\rm ord}_p(d_i).$$
In what follows, we first compute the possible $d(T)$ and then the number  
of $T$ with a given $d(T)$ by using Mathematica (version 12.1). The data $\{\tau_p(T)\}_p$ is immediately read off from 
$d(T)$.

Now we write 
$$T=\begin{pmatrix}
2 & x& y \\
\bar{x} & 2 &z \\
\bar{y}  & \bar{z}  & 2
\end{pmatrix},\ x,y,z\in \frak o'.$$
Then we have 
$$T\times T=\begin{pmatrix}
4-N(z) & y\bar{z}-2x&  xz-2y \\
z\bar{y}-2\bar{x} & 4-N(y) & \bar{x}y-2z \\
\bar{z}\bar{x}-2\bar{y}  & \bar{y}x-2\bar{z}  & 4-N(x)
\end{pmatrix}$$
and 
$$\det(T)=8-2(N(x)+N(y)+N(z))+t,\ t:={\rm tr}((xz)\bar{y}).$$
Since $T>0$ and $x,y,z$ are integral (imaginary) octonions, we have 
$$0\le N(x),N(y),N(z)\le 3.$$
Note that $N(x)=0$ implies $x=0$. 

\begin{lemma}\label{tr} Let $x,y\in \o\otimes_\Z\R$.
Then, $|{\rm tr}(x\bar y)|\le 2\sqrt{N(x)N(y)}$. 
\end{lemma}
\begin{proof}Write $x=\ds\sum_{i=0}^7 x_i e_i,\ y=\sum_{i=0}^7 y_i e_i$ 
for $x_i, y_i\in \R$ ($0\le i\le 7$) by using the basis in \cite[Section 2]{KY}. 
It is easy to see that 
$${\rm tr}(x\bar y)=2\Big(x_0y_0+\sum_{i=1}^7 x_i y_i\Big).
$$
Since $N(x)=\ds\sum_{i=0}^7 x^2_i,\ N(y)=\ds\sum_{i=0}^7 y^2_i$, 
we have the desired inequality by
Cauchy-Schwarz inequality.
\end{proof}

By the equality criterion in Cauchy-Schwarz inequality, we have
\begin{corollary} \label{Trace} For octonions $x,y$, we have ${\rm tr}(x\bar y)= 2\sqrt{N(x)N(y)}$ if and only if $y=x$. In particular, if $N(x)=N(y)=1$,
${\rm tr}(x\bar y)= 2$ if and only if $y=x$.
\end{corollary}

Since $N(xy)=xy\overline{xy}=x(y\bar{y})\bar{x}=N(y)x\bar{x}=N(y)N(x)=N(x)N(y)$ and 
$N(x)=N(\bar{x})$, we have the following.
\begin{corollary}\label{ineq}For $x,y,z\in \o$, it holds that 
$$|{\rm tr}((xz)\bar{y})|\le 2\sqrt{N(x)N(y)N(z)}.$$ 
\end{corollary}

Let us turn to our setting. Then,  
$$\det(T)=8-2(N(x)+N(y)+N(z))+t>0,\ (t:={\rm tr}((xz)\bar{y}))$$
and 
$$0\le N(x),N(y),N(z)\le 3.$$
By Corollary \ref{ineq}, we have $|{\rm tr}((xz)\bar{y})|\le 10$ since $2\sqrt{3^3}=10.3\cdots$. 
Let us consider the set 
$$S:=\Bigg\{(n_1,n_2,n_3,t,8-2(n_1+n_2+n_3)+t)\in \Z^5\ \Bigg|\ 
\begin{array}{l}
0\le n_1,n_2,n_3\le 3,\\ 
|t|\le 2\sqrt{n_1n_2n_3},\\  
8-2(n_1+n_2+n_3)+t>0
\end{array}
\Bigg\}.$$
Then, $|S|=39$ and explicitly, 
$$
S=\{
   (0, 0, 0, 0, 8), (0, 0, 1, 0, 6),(0, 0, 2, 0, 4), (0, 0, 3, 0, 2), (0, 1, 0, 0, 6), (0, 1, 1, 0, 4),$$
$$(0, 1, 2, 0, 2), (0, 2, 0, 0, 4),(0, 2, 1, 0, 2), (0, 3, 0, 0, 2), (1, 0, 0, 0, 6), (1, 0, 1, 0, 4),$$ 
$$(1, 0, 2, 0, 2), (1, 1, 0, 0, 4),(1, 1, 1, -1, 1), (1, 1, 1, 0, 2),(1, 1, 1, 1, 3), (1, 1, 1, 2, 4),$$
$$(1, 1, 2, 1, 1), (1, 1, 2, 2, 2),(1, 1, 3, 3, 1), (1, 2, 0, 0, 2), (1, 2, 1, 1, 1), (1, 2, 1, 2, 2), $$
$$(1, 2, 2, 3, 1), (1, 2, 2, 4, 2),(1, 3, 1, 3, 1), (2, 0, 0, 0, 4), (2, 0, 1, 0, 2), (2, 1, 0, 0, 2), $$
$$(2, 1, 1, 1, 1), (2, 1, 1, 2, 2),(2, 1, 2, 3, 1), (2, 1, 2, 4, 2), (2, 2, 1, 3, 1), (2, 2, 1, 4, 2), $$
$$(2, 2, 2, 5, 1), (3, 0, 0, 0, 2),(3, 1, 1, 3, 1)\}.$$
It follows $d_3\in \{1, 2, 3, 4, 6, 8\}$. 
Since the symmetric group $\frak S_3$ of degree 3 acts on the first three coordinates of 
$S$ by  $(n_1,n_2,n_3)\mapsto (n_{\sigma(1)},n_{\sigma(2)},n_{\sigma(3)})$ for 
$\sigma\in\frak S_3$, we have $|S/\frak S_3|=16$ and explicitly
$$
S/\frak S_3=
\{ (0, 0, 0, 0, 8), (1, 0, 0, 0, 6), (2, 0, 0, 0, 4), (3, 0, 0, 0, 2), (1, 1, 0, 0, 4), (2, 1, 0, 0,2),$$
$$(1, 1, 1, -1, 1), (1, 1, 1, 0, 2), (1, 1, 1, 1, 3), (1, 1, 1, 2, 4), (2, 1, 1, 1, 1), (2, 1, 1, 2,2),$$
$$(3, 1, 1, 3, 1), (2, 2, 1, 3, 1), (2, 2, 1, 4, 2), (2, 2, 2, 5,1)\}.
$$
  
We are now ready to carry out the computation. 
For a given $(n_1,n_2,n_3,t,8-2(n_1+n_2+n_3)+t)\in S/\frak S_3$,  
we shall count $T=\begin{pmatrix}
2 & x& y \\
\bar{x} & 2 &z \\
\bar{y}  & \bar{z}  & 2
\end{pmatrix},\ x,y,z\in \frak o'$ with 
$$N(x)=n_1,\ N(y)=n_2,\ N(z)=n_3,\ t={\rm tr}((xz)\bar{y}).$$ 
Put 
$$S_{(n_1,n_2,n_3)}(t):=\{(x,y,z)\in \frak o'\ |\ 
N(x)=n_1,\ N(y)=n_2,\ N(z)=n_3,\ t={\rm tr}((xz)\bar{y}\}$$
for simplicity. Notice that $|S_{(n_1,n_2,n_3)}(t)|$ is stable under the action of 
$\frak S_3$ and that $|S_{(n_1,n_2,n_3)}(t)|=|S_{(n_1,n_2,n_3)}(-t)|$ since $-1\in \frak o$. 
Let us observe that when one of $x,y,z$ is non-zero, say $x$, if $N(x)<4$, then 
$\gcd(x)=1$. Thus $d_1=1$ and $\tau(1)=\ord_p(d_1)=0$ for all prime $p$. 

\noindent
\textbf{Case 1}. When $(n_1,n_2,n_3)=(0,0,0)$, then, $x=y=z=t=0$. 
In this case, $T=\diag(2,2,2)$ and $d(T)=(2,4,8)$. 
Then, $$|S_{(0,0,0)}(0)|=1.$$

\noindent
\textbf{Case 2}. When $(n_1,n_2,n_3)=(1,0,0)$  then, $y=z=t=0$. As mentioned before, $d_1=1$ since $x\neq 0$. 
Since $4-N(x)=3$ is odd and 
$$T\times T=\begin{pmatrix}
4 & -2x&  0 \\
-2\bar{x} & 4 & 0 \\
0 & 0 & 3
\end{pmatrix},$$
we have $d_2=1$. Further, $d_3=6$. Thus, $d(T)=(1,1,6)$. 
Then, 
$$|S_{(1,0,0)}(0)|=N_{{\rm ioc}}(1)=126.$$

\noindent
\textbf{Case 3}. When $(n_1,n_2,n_3)=(3,0,0)$  then, $y=z=t=0$. 
Similar to the previous case, we have $d(T)=(1,1,2)$. 
Then, 
$$|S_{(3,0,0)}(0)|=N_{{\rm ioc}}(3)=2072.$$

\noindent
\textbf{Case 4}. When $(n_1,n_2,n_3)=(2,0,0)$, then $y=z=t=0$. 
We see $d_1=1$ as before and $d_2=2$ from 
$$T\times T=\begin{pmatrix}
4 & -2x&  0 \\
-2\bar{x} & 4 & 0 \\
0 & 0 & 2
\end{pmatrix}$$
since $\gcd(x)=1$. Further, $d_3=8-4=4$. 
Thus, $d(T)=(1,2,4)$. 
Then, 
$$|S_{(2,0,0)}(0)|=N_{{\rm ioc}}(2)=756.$$

\noindent
\textbf{Case 5}. When $(n_1,n_2,n_3)=(1,1,0)$, then $z=t=0$. 
In this case, $d_1=1$ and $d_2=1$ since $\gcd(4-N(z),4-N(x))=\gcd(4,3)=1$. 
Further, $d_3=8-2(1+1)=4$. Hence $d(T)=(1,1,4)$. 
Then, 
$$|S_{(1,1,0)}(0)|=N_{{\rm ioc}}(1)^2=126^2=15876.$$

\noindent
\textbf{Case 6}. When $(n_1,n_2,n_3)=(2,1,0)$, then $z=t=0$. 
As in the previous case, we have $d(T)=(1,1,2)$.
Then, 
$$|S_{(2,1,0)}(0)|=N_{{\rm ioc}}(1)N_{{\rm ioc}}(2)=126\times 756=95256.$$

\noindent
\textbf{Case 7}. When $(n_1,n_2,n_3)=(1,1,1)$, we have $t=-1,0,1,2$. 
We remark that 
$$T\times T=\begin{pmatrix}
3 & y\bar{z}-2x&  xz-2y \\
z\bar{y}-2\bar{x} & 3 & \bar{x}y-2z \\
\bar{z}\bar{x}-2\bar{y}  & \bar{y}x-2\bar{z}  & 3
\end{pmatrix}$$
and 
$$d_3=\det(T)=2+t,\ N(y\bar{z}-2x)=5-2t.$$ 
Thus, $d_3$ and $N(y\bar{z}-2x)$ are coprime unless $t=1$ in which case 
$d_3=N(y\bar{z}-2x)=3$. However, $\gcd(y\bar{z}-2x)=1$ when $N(y\bar{z}-2x)=3$.  
Thus, in any case, we have $d_1=d_2=1$ and 
$$d(T)=(1,1,2+t),\ t=-1,0,1,2.$$
By Mathematica (version 12.1), we have 
$$|S_{(1,1,1)}(-1)|=459648,\ |S_{(1,1,1)}(0)|=1065960,\ |S_{(1,1,1)}(1)|=459648
,\ |S_{(1,1,1)}(2)|=7560.$$

\noindent
\textbf{Case 8}. When $(n_1,n_2,n_3)=(2,1,1)$, then $t=1,2$. 
In this case, $d_1=d_2=1=1,\ d_3=t$ and thus $d(T)=(1,1,t)$. 
By Mathematica (version 12.1), we have 
$$|S_{(2,1,1)}(1)|=3193344,\ |S_{(2,1,1)}(2)|=671328.$$

\noindent
\textbf{Case 9}. When $(n_1,n_2,n_3)=(3,1,1)$, then $t=3$. 
In this case, $d_1=d_2=1=1,\ d_3=8-10+3=1$ and thus $d(T)=(1,1,1)$. 
By Mathematica (version 12.1), we have 
$$|S_{(3,1,1)}(3)|=249984.$$

\noindent
\textbf{Case 10}. When $(n_1,n_2,n_3)=(2,2,1)$, then $t=3,4$. 
In this case, $d_1=d_2=1=1,\ d_3=8-10+t=t-2$ and thus $d(T)=(1,1,t-2)$. 
By Mathematica (version 12.1), we have 
$$|S_{(2,2,1)}(3)|=2032128,\ |S_{(2,2,1)}(4)|=31752.$$

\noindent
\textbf{Case 11}. When $(n_1,n_2,n_3)=(2,2,2)$, then $t=5$. 
In this case, $d_1=1$ as before. Notice that 
$$T\times T=\begin{pmatrix}
2 & y\bar{z}-2x&  xz-2y \\
z\bar{y}-2\bar{x} & 2 & \bar{x}y-2z \\
\bar{z}\bar{x}-2\bar{y}  & \bar{y}x-2\bar{z}  & 2
\end{pmatrix}.$$
Since $\gcd(2, y\bar{z}-2x)=\gcd(2, y\bar{z})=\gcd(2, y)\gcd(2,\bar{z})=1$ because 
$\gcd(y)=\gcd(z)=1$. Thus, $d_2=1$. Further, $d_3=8-12+5=1$ and therefore, 
$d(T)=(1,1,1)$. 
The number of $T$ with $d(T)=(1,1,1)$ is given by the cardinality of 
$$S_{(2,2,2)}(5):=\{(x,y,z)\in \frak o'^3\ |\ N(x)=N(y)=N(z)=2,\ 
{\rm tr}((xz)\bar{y})=5\}$$
and by Mathematica (version 12.1), we have 
$$|S_{(2,2,2)}(5)|=1306368.$$ 

\noindent
\textbf{Summary}.  For each given $(d_1,d_2,d_3)\in \Z^3_{>0}$, let us count all 
$T\in \frak J(\Z)_{>0}$ such that $T_1=\diag(2,2,2)$ and $d(T)=(d_1,d_2,d_3)$. 

\medskip
{\small
\begin{center}
\begin{tabular}{|l|l|l|l|l|l|l|l|} \hline
$(d_1,d_2,d_3)$ & $(1,1,1)$ & $(1,1,2)$ & $(1,1,3)$ & $(1,1,4)$ & $(1,1,6)$ & $(1,2,4)$ & $(2,4,8)$ \\ \hline
  $\#$ of $T$  & 18192384 & 3752952 &459648 & 55188 & 378& 2268 & 1 \\ \hline
 Cases &Case 9, Case 11 & Case 3, Case 6 & Case 7 ($t=1$)  &  Case 5 & Case 2 & Case 4  & Case 1 \\ 
  &Case 7 ($t=-1$)  & Case 7 ($t=0$) &  & Case 7 ($t=2$)  &  &   &  \\  
   &Case 8 ($t=1$)  & Case 8 ($t=2$) & &   &  &   &  \\  
     &Case 10 ($t=3$)   & Case 10 ($t=4$)  &  &  &  &   &  \\  \hline
 \end{tabular}
  \captionof{table}{}
 \end{center}
}

\medskip
We remark that we need to take the number of $\frak S_3$-orbits into account. 
For example, let us consider the case $(d_1,d_2,d_3)=(1,1,2)$. Then, 
\begin{enumerate}
\item Case  3 ($(n_1,n_2,n_3)=(3,0,0)$), 
\item Case 6 ($(n_1,n_2,n_3)=(2,1,0)$), 
\item Case 7 ($t=0$)($(n_1,n_2,n_3)=(1,1,1)$), 
\item Case 8 ($t=2$) ($(n_1,n_2,n_3)=(2,1,1)$), and  
\item Case 10 ($t=4$) ($(n_1,n_2,n_3)=(2,2,1)$)
\end{enumerate}
contribute to the total amount and it is given by  

$$|\frak S_3(3,0,0)|\cdot |S_{(3,0,0)}(0)|
+|\frak S_3(2,1,0)|\cdot |S_{(2,1,0)}(0)|
+|\frak S_3(1,1,1)|\cdot |S_{(1,1,1)}(0)|
$$
$$+|\frak S_3(2,1,1)|\cdot |S_{(2,1,1)}(2)|
+|\frak S_3(2,2,1)|\cdot |S_{(2,2,1)}(4)|$$
$$=3N_{{\rm ioc}}(3)+6N_{{\rm ioc}}(1)N_{{\rm ioc}}(2)+1\cdot 1065960
+3\cdot 671328
+3\cdot 31752=3752952.
$$

\section{Appendix B: Arthur's classification for holomorphic Siegel modular forms on 
$Sp_6$}
In this appendix, by using Arthur's classification, we show that all Hecke eigen holomorphic Siegel modular forms on $\Bbb H_3$ with level one and the scalar weight $k\ge 4$ are either of Miyawaki type I, II, or genuine forms. We refer to \cite{Atobe} or \cite{CL} 
for Arthur's classification for the symplectic groups. As for Miyawaki lifts, they have been already discussed in \cite{Atobe}, but there was not a  discussion of their relationship to CAP representations.  
We also remark that no endoscopic forms show up. This phenomena is a special feature in 
the case of level one.  

Let us recall some notations (see also \cite[Section 4]{KWY}). A (discrete) global A-parameter for $Sp_{2n}$ is a symbol   
\begin{equation}\label{para}
\psi=\pi_1[e_1]\boxplus \cdots\boxplus \pi_r[e_r]
\end{equation}
satisfying the following conditions:
\begin{enumerate}[{\rm A}-(1)]
\item for each $i$ $(1\le i \le r)$, $\pi_i$ is an irreducible unitary cuspidal self-dual automorphic representation  of ${\rm GL}_{m_i}(\A)$. 
In particular, the central character $\omega_i$ of $\pi_i$ is trivial or quadratic; 
\item for each $i$, we have $e_i\in \Z_{>0}$ and $\ds\sum_{i=1}^r m_ie_i=2n+1$;
\item if $e_i$ is odd, $\pi_i$ is orthogonal, i.e., $L(s,\pi_i,{\rm Sym}^2)$ has a pole at $s=1$;
\item if $e_i$ is even, $\pi_i$ is symplectic, i.e., $L(s,\pi_i,\wedge^2)$ has a pole at $s=1$;
\item $\omega^{e_1}_1\cdots \omega^{e_r}_r=1$; 
\item if $i\neq j$, $\pi_i\simeq \pi_j$, then $e_i\neq e_j$. 
\end{enumerate}

Let $\Psi(Sp_{2n})$ be the set of global $A$-parameters for $Sp_{2n}$.
We say that two global $A$-parameters $\boxplus_{i=1}^r\pi_i[e_i]$ and $\boxplus_{i=1}^{r'}\pi'_i[e'_i]$ are 
equivalent if $r=r'$ and there exists $\sigma\in \mathfrak S_r$ such that $e'_i=e_{\sigma(i)}$ and 
$\pi'_i=\pi_{\sigma(i)}$. 
For each $\psi \in \Psi(Sp_{2n})$, one can associate the global A-packet $\Pi_\psi$. 
\begin{defin}
Let $\psi=\ds\boxplus_{i=1}^r\pi_i[e_i]$ be a global A-parameter. 
Then, 
\begin{itemize}
\item $\psi$ is said to be semi-simple if $e_1=\cdots=e_r=1$; otherwise, $\psi$ is said to be non-semi-simple; 
\item $\psi$ is said to be simple if $r=1$ and $e_1=1$.   
\end{itemize}
\end{defin}

\begin{defin}
Let $\psi=\ds\boxplus_{i=1}^r\pi_i[e_i]$ be a global A-parameter. 
Let $\Pi$ be a cuspidal automorphic representation of $\Sp_6(\A)$ which belongs to $\Pi_\psi$.  
Then, 
\begin{itemize}
\item $\Pi$ is said to be endoscopic if $\psi$ is semi-simple and $r>1$;
\item $\Pi$ is said to be genuine if $\psi$ is simple.   
\end{itemize}

We call a Hecke eigen form $F$ an endoscopic form or a genuine form 
if its associated representation $\Pi_F$ is endoscopic or genuine, respectively. 
\end{defin}

\begin{remark} Let $F$ be a Hecke eigenform for $\Bbb H_n$, and $\Phi: \Sp_{2n}(\Bbb A)\longrightarrow \Bbb C$ be the adelic function attached to $F$. The referee pointed out that the level one assumption on $F$ guarantees that $\Phi$ generates an irreducible automorphic representation of $\Sp_{2n}(\Bbb A)$ due to \cite[Corollary 3.4]{NPS}. Then one can speak of the associated representation $\Pi_F$.
\end{remark}

\begin{defin}\label{cap}(cf. \cite[Section 1]{GG}) A cuspidal automorphic representation $\Pi$ of $\Sp_6(\A)$ 
is said to be a CAP (cuspidal associated to a parabolic subgroup) representation if  
there exists a parabolic subgroup $P=MN$, a unitary cuspidal automorphic 
representation $\pi$ of $M(\A)$ and $s>0$, such that $\Pi$ is nearly equivalent to 
${\rm Ind}^{\Sp_6(\A)}_{P(\A)} \pi\delta_P^s$ (the normalized induction), where $\delta_P$ is the modulus character of $P$.


A cuspidal Hecke eigen automorphic form $F$ on $\Sp_6(\A)$ is said to be a CAP 
form if the corresponding representation $\Pi_F$ is a CAP representation.  
\end{defin}

Let $F$ be a Hecke eigen cusp form on $\Sp_6(\A)$ of level 1 with the scalar weight $k\ge 4$  and 
$\Pi=\Pi_\infty\otimes \otimes'_p \Pi_p$ be the corresponding cuspidal representation 
of $\Sp_6(\A)$. Since $k\ge 4$, $\Pi_\infty$ is a discrete series representation. 
Let $\psi$ be the global Arthur parameter such that $\Pi\in \Pi_\psi$.
Such $\psi$ is uniquely determined by $L^S(s,\Pi,{\rm St})$ for 
any finite set $S$ of primes. 

\subsection{CAP representations of $Sp_6$}
\begin{prop}\label{cap}Keep the notation as above. 
It holds that
\begin{enumerate}
\item if $\psi$ is semisimple, then $\Pi$ is tempered everywhere; 
\item if $\Pi$ is a CAP representation, then $\psi$ is non-semisimple.
\end{enumerate}
\end{prop}
\begin{proof}The first claim is well-known (see \cite[the proof of Theorem 4.3]{KWY}). 
As for the second claim, for the 
global A-parameter $\psi$ such that $\Pi\in \Pi_\psi$, 
if we write 
$\psi=\pi_1[e_1]\boxplus \cdots\boxplus \pi_r[e_r]$, 
we shall prove that $e_i>1$ for some $1\le i\le r$.  
Note that by \cite[the proof of Theorem 4.3]{KWY} again (or \cite[p.201-202, the proof of 
Corollary 8.2.19]{CL} for details), all $\pi_i$ are tempered everywhere. 
Further, 
$$L_p(s,\Pi,{\rm St})=\prod_{i=1}^r L_p(s,\pi_i[e_i])
$$
for all but finitely many rational primes $p$. Here $L_p(s,\pi[e])=
\prod_{j=0}^{e-1} L_p(s+\frac {e-1}2-j,\pi)$.

If $\Pi$ is nearly equivalent to ${\rm Ind}^{\Sp_6(\A)}_{P(\A)} \pi\delta_P^s$, $s>0$, then $\Pi_p$ is non-tempered for almost all $p$. 
Hence $e_i>1$ for some $i$.

For example, if $\Pi$ is nearly equivalent to ${\rm Ind}^{\Sp_6(\A)}_{P(\A)} \pi |\det|^\frac 12\otimes\sigma$, where
$P=MN$, $M\simeq GL_2\times SL_2$, and $\pi$ is a cuspidal representation of $GL_2$, and $\sigma$ is a cuspidal representation of $SL_2$. Then 
$$L_p(s,\Pi_p,{\rm St})=L(s+\frac 12,\pi_p)L(s-\frac 12,\pi_p)L(s,{\rm Sym}^2\sigma_p'),
$$
where $\sigma'$ is a cuspidal representation of $GL_2$ such that $\sigma'|_{SL_2}=\sigma$.
\end{proof}

\begin{remark}\label{pAp}Keep the notation as above.  
Assume $\Pi$ belongs to the global A-packet for $\psi$ non-semisimple. 
In particular, if we write $\psi=\pi_1[e_1]\boxplus \cdots\boxplus \pi_r[e_r]$ as in (\ref{para}), then it satisfies the following:
\begin{enumerate}
\item There exists $e_i\ (1\le i\le r)$ such that $e_i\ge 2$ (say $e_1\ge 2$). 
\item For each $ 1\le i\le r$, $\pi_i$ is regular algebraic 
self-dual, and unramified everywhere (cf. \cite[the proof of Corollary 8.2.19]{CL}). Thus, if $m_i=1$, then 
$\pi_i$ has to be the trivial character on $\A^\times$. 
Therefore, the condition A-(6) yields $\#\{i\ |\ (m_i,e_i)=(1,1)\}\le 1$. Further, 
when $m_i=2$, $e_i$ has to be even because $\pi_i$ is unramified everywhere, and hence $L(s,\pi_i,{\rm Sym}^2)$ has no pole at $s=1$.
It follows that $r\le 3$. 
Note that if $m_i$ is odd, then $e_i$ has to be odd since $L(s,\pi_i,\wedge^2)$ has no pole at $s=1$.
\item The localization of $\psi$ at the archimedean place is given by 
$$\psi_\infty=\bigoplus_{i=1}^r \rho_{\pi_i,\infty}\otimes S_{e_i}:W_\R \times \SL_2(\C)\lra 
\GL_{7}(\C)$$
where $\rho_{\pi_{i,\infty}}:W_\R\lra \GL_{m_i}(\C)$ is the local Langlands parameter 
of $\pi_{i,\infty}$, and $S_d$ stands for the unique irreducible algebraic representation of $\SL_2(\C)$ of dimension $d$. 

\end{enumerate}
\end{remark}

Let ${\bf k}=(k,k,k)$ be the lowest weight of $\Pi_\infty$. By 
\cite[Proposition 3.2]{Atobe} (note that the holomorphy assumption is important here), 
the localization of $\psi$ at the archimedean place can be written as 
$$\psi_\infty=\Big(\bigoplus_{i=1}^t \rho_{\alpha_i}\otimes S_{d_i}\Big)\oplus \sgn$$
such that 
\begin{itemize}
\item $\alpha_i,d_i\in \Z_{>0}$ for $1\le i\le t$;
\item $\alpha_1>\cdots > \alpha_t>0,\ \ds\sum_{i=1}^t d_i=3$;
\item $\alpha_i+d_i\equiv 1$ mod 2;
\item $\bigcup_{i=1}^t \Big\{\frac{\alpha_i+d_i-1}{2},\ldots,\frac{\alpha_i-d_i+1}{2}\Big\}
=\{k-1,k-2,k-3\}$.
\end{itemize}
By the definition of non-semi-simpleness, there exists $1\le i\le r$ such that $e_i>1$ and 
it yields that there exists $1\le i\le t$ such that $d_i>1$. 
Thus, it follows that $1\le t \le 2$. 

When $t=1$, $\psi$ has to be $\pi[3]\oplus \sgn$ for some orthogonal cuspidal 
representation $\pi$ of $\GL_2(\A_\Q)$. Since $\Pi$ is of level one, it is impossible. 

When $t=2$, we have $d_1=1,\ d_2=2$ or $d_1=2,\ d_2=1$. 
For the former case, we have $\alpha_1=2(k-1),\ \alpha_2=2k-5$ and thus  
$$\psi_\infty=\rho_{2(k-1)}\otimes S_2\oplus \rho_{2k-5}\oplus {\rm sgn}.
$$

Let $\psi=\pi_1[d_1]\boxplus \cdots\boxplus \pi_r[d_r]$ be the corresponding global A-parameter with $1\le r\le 3$. 
Note that $(\pi_i[d_i])_\infty=S_{d_i}$ if $\pi_i=1$ and $1[d]$ can not occur in $\psi$ 
when $d$ is even. 

When $r=3$, the only possible case is  
$$\psi=1\boxplus 1[3] \boxplus \pi,
$$ 
where 
$\pi$ is a cuspidal representation of $\GL_3(\A)$. 
However, $\psi_\infty$ cannot have the trivial character and so this case cannot occur. 

When $r=1$, $\psi=1[7]$ (since we have assumed $\psi$ is non-semisimple) so that 
$\psi_\infty=1\otimes S_7$ which is impossible to occur. 

When $r=2$, the possible cases for $\psi$ are 
$$\pi_1[2]\boxplus \pi_2,\quad \pi_3\oplus 1[3]
$$
where $\pi_1,\pi_2,\pi_3$ are cuspidal representations of 
$\GL_n(\A)$ for $n=2,3,4$ respectively. The second case is excluded since 
$\psi_\infty$ does not include $1\otimes S_3$. 
As for the first case, $\pi_2$ has to be orthogonal and hence it is self-dual. By Ramakrishnan \cite{Ra}, it is of the form ${\rm Sym}^2\sigma$, where $\sigma$ is a   
cuspidal representation of $GL_2(\A)$.  
Thus, by arranging symbols, we have 
\begin{equation}\label{typeI}
\psi={\rm Sym}^2\sigma\oplus \pi[2],
\end{equation}
such that $\sigma$ and $\pi$ come from cusp forms in $S_k(\SL_2(\Z))$ and 
$S_{2k-4}(\SL_2(\Z))$ respectively. The cuspidal representation with this Arthur 
parameter is called Miyawaki lift of type I.

For the latter case ($d_1=2,d_2=1$), we have $\alpha_1=2k-3,\ \alpha_2=2(k-3)$. 
In this case, we have 
\begin{equation}\label{typeII}
\psi=\pi[2]\oplus {\rm Sym}^2\sigma,
\end{equation}
such that $\pi$ and $\sigma$ come from cusp forms in $S_{k-2}(\SL_2(\Z))$ and 
$S_{2k-2}(\SL_2(\Z))$ respectively. The cuspidal representation with this Arthur 
parameter is called Miyawaki lift of type II.

Summing up, we have proved the following:

\begin{theorem}\label{CAPMiyawaki}
Let $F$ be a CAP Hecke eigen holomorhic Siegel modular form on $\Sp_6(\A)$ which is of level one and of 
scalar weight $k\ge 4$. Then, $F$ is either Miyawaki lift of type I  or 
Miyawaki lift of type II. Conversely, the two kinds of Miyawaki lifts are CAP forms. 
More precisely, if $F$ is of Miyawaki lift of type I $($resp. type II$)$, then the 
associated cuspidal representation $\Pi_F$ is nearly equivalent to 
$${\rm Ind}_P^G\, \pi_f|\det|^{\frac{1}{2}}\otimes \pi_g$$
where $P=MN$ with $M=GL_2\times SL_2$ and 
$\pi_f$ is the cuspidal representation of $\GL_2(\A)$ 
associated to a newform $f\in S_k(\SL_2(\Z))$ $($resp. $f\in S_{k-2}(\SL_2(\Z)))$ while 
$\pi_g$ is the cuspidal representation of $\SL_2(\A)$ 
associated to a newform $g\in S_{2k-4}(\SL_2(\Z))$  $($resp. $g\in S_{2k-2}(\SL_2(\Z)))$.    
\end{theorem} 

\subsection{Endoscopic representations of $Sp_6$} 
Let $F$ be a Hecke eigen endoscopic holomorhic Siegel modular form on $\Sp_6(\A)$ which is of level one and of scalar weight $k\ge 4$. Let $\psi=\ds\boxplus_{i=1}^r\pi_i$ be 
the global A-parameter such that $\Pi_F$ belongs to $\Pi_\psi$. Since $F$ is of level one, 
by a similar argument in the previous subsection, 
the only possible parameter is of form $\psi=\sigma_4\boxplus \sigma_3$ where 
$\sigma_n$ is an orthogonal cuspidal representation of $\GL_n(\A_\Q)$. Since $\sigma_4$ is unramified at all finite places, the central character is trivial, and hence it is a transfer from a cuspidal representation of $SO(2,2)\simeq (SL_2\times SL_2)/\mu_2$. Hence by \cite{Ra0}, $\sigma_4=\pi_{k_1}\boxtimes \pi_{k_2}$, where $k_1,k_2\in \Z_{>0},\ k_1>k_2$, and 
$\pi_{k_i}$ is a cuspidal representation of $GL_2$ associated to elliptic newform of level 1 and of weight $k_i+1$. 
As in the previous subsection, $\sigma_3={\rm Sym}^2 \pi_{k_3}$, where $\pi_{k_3}$ is a cuspidal representation of $GL_2$ associated to elliptic newform of level 1 and of weight $k_3+1$. 
Therefore, we can rewrite $\psi$ as $\psi=(\pi_{k_1}\boxtimes \pi_{k_2})\boxplus {\rm Sym^2}\pi_{k_3}.$
Then, $\psi_\infty=\rho_{k_1+k_2}\oplus \rho_{k_1-k_2}\oplus 
\rho_{2k_3}\oplus \sgn$. 
Applying \cite[Proposition 3.2]{Atobe} for $r=3$ and $d_1=d_2=d_3=1$, we have 
$$\Big\{\frac{k_1+k_2}{2},\ \frac{k_1-k_2}{2},\ k_3\Big\}=\{k-1,k-2,k-3\}.$$
It follows $k_2\le 2$. Since there is no elliptic cusp form of weight $k_2+1$ and of level 1, 
we have a contradiction. Thus, there are no holomorphic endoscopic forms of level 1 on $\Sp_6(\A)$. 

In conclusion, we have proved the following:
\begin{theorem}\label{AllHecke}
Let $F$ be a Hecke eigen holomorhic Siegel modular form on $\Sp_6(\A)$ which is of level one and of 
scalar weight $k\ge 4$. Then, $F$ is either a Miyawaki lift of type I,II $($a CAP form$)$, or a genuine form.    
\end{theorem} 

\noindent{\bf Data Availability Statement} Data will be made available upon reasonable request.

\noindent{\bf Conflict of interest} We have no conflicts of interest to disclose.

\end{document}